\documentclass[11pt,a4paper]{article}

\usepackage[utf8]{inputenc}
\usepackage[T1]{fontenc}
\usepackage{amsmath,amsthm,amsopn,amsfonts,amssymb,dsfont,color}
\usepackage{mathrsfs}
\usepackage{tikz}
\usepackage{xcolor}
\usepackage{graphicx}
\usepackage{stmaryrd}
\usepackage{enumerate}
\usepackage{appendix}
\usepackage{xcolor}
\usepackage{hyperref}

\usepackage{a4,a4wide}

\def\ep{{\varepsilon}}
\def\R{\mathbb R}

\newcommand{\norm}[2]{\lvert\lvert {#1} \rvert \rvert_{{#2}}}

\newtheorem{theorem}{\textbf{Theorem}}[section]

\newtheorem{lemma}[theorem]{\textbf{Lemma}}
\newtheorem{proposition}[theorem]{\textbf{Proposition}}

\newtheorem{corollary}[theorem]{\textbf{Corollary}}

\newtheorem{remark}[theorem]{\textbf{Remark}}

\newtheorem{example}[theorem]{\textbf{Example}}

\numberwithin{equation}{section}

\title{\bf Adaptation  in shifting and size-changing environments under selection}
\author{ Matthieu Alfaro, Adel Blouza and Nessim Dhaouadi\vspace{5pt}\\
Univ. Rouen Normandie, CNRS, LMRS UMR 6085\\ F-76000 Rouen, France.}
\date{}

\begin{document}
\maketitle
\vspace{10pt}
\begin{abstract} 
We propose a model to characterize how a diffusing population adapts under  a time periodic selection, while its  environment undergoes shifts and size changes, leading to significant differences with classical results on fixed domains. After studying the underlying periodic parabolic principal eigenelements, we address the {\it extinction vs. persistence}  issue, taking into account the interplay between the moving habitat and periodic selection. Subsequently, we employ a space-time finite element approach, establish the well-posedness of the approximation scheme, and conduct numerical simulations to explore these dynamics.

\vspace{20pt}
\noindent{Key Words: dynamics of adaptation, long time behavior, periodic parabolic eigenelements, finite elements approximation, numerical simulation.}\\

\noindent{AMS Subject Classifications:  35K57 (Reaction-diffusion equations), 35P15 (Estimates of eigenvalues in context of PDEs), 65M06 (Finite difference methods for initial value and initial-boundary value problems involving PDEs).}
\end{abstract}
 
\section{Introduction}\label{s:intro}

In this work, we consider the solutions $u=u(t,x)$ to the boundary value problem
\begin{equation}\label{eq}
\begin{cases}
u_t=du_{xx}+\left(r-\frac{\alpha(t)}{2}\left(x-x_{opt}(t)\right)^2\right)u,  & \quad t>0,\, A(t)<x<A(t)+L(t),\vspace{5pt}\\
u(t,A(t))=u(t,A(t)+L(t))=0, & \quad t>0,
\end{cases}
\end{equation}
where $d>0$, $r>0$,
\begin{equation}\label{x-opt}
x_{opt}(t):=A(t)+\beta(t) L(t),
\end{equation}
and, obviously, we use the shortcuts $u_t$, $u_{xx}$ for $\partial_tu$, $\partial_{xx}u$. One of the originalities  of \eqref{eq} stands in the fact that it is posed on an interval that may not only shift, through the function $A\in C^2([0,+\infty);\R)$, but also change size, through the function $L\in C^2([0,+\infty);(0,+\infty))$. Both functions $\alpha$ and $\beta$ are assumed to be Hölder continuous and $T$-periodic for some given $T>0$, and $\alpha$ is positive. The boundary conditions at $x=A(t)$, $x=A(t)+L(t)$ are of the Dirichlet type. We aim at exploring the long time behavior of the solutions to \eqref{eq}. 

\medskip

In ecology or population dynamics, this serves as a linear model to describe the adaptation of a diffusing population, whose mobility is measured by the constant $d>0$, living in a moving and size changing habitat. The reasons for such moving range boundaries, see \cite{All-1}, could be the consequences of flooding, forest fire, etc. This also connects to the issue of ecological  niches shifted by some external factors, such as {\it Global Warming}, and which has  received a lot of attention, 
see  \cite{Pot-Lew-04}, \cite{Roq-Roq-Ber-Kre-08}, \cite{Ber-Die-Nag-Zeg-09}, \cite{Bou-Nad-15}, \cite{Alf-Ber-Rao-17}, \cite{Bou-Gil-19} and the references therein. Furthermore, the growth depends on the location in the environment: the maximal growth rate $r>0$ is reached at the time dependent position \eqref{x-opt}, that turns out to be outside the  domain when $\beta(t)\not \in (0,1)$, and, away from this optimal position, the growth decays quadratically  with a pressure measured by the function $\alpha$. The periodicity of functions $\alpha$ and $\beta$ may reflect, see \cite{All-3},  some seasonal variations in temperature, water level, etc.

On the other hand, in the context of evolutionary biology,  diffusion then models the mutation process, where $x$ denotes a phenotypic trait (see \cite{Kim-65}, \cite{Lan-75}, \cite{Bur-00-book}  among many other references). In this framework, the fitness  (reproductive success) of a phenotype $x$ is described by a function that decreases away from the optimum and, here, we use  Fisher's geometrical phenotype-to-fitness model, see \cite{Ten-14}, \cite{Mar-Len-15}, \cite{Alf-Car-17}, \cite{Ham-Lav-Mar-Roq-20}, \cite{Ham-Lav-Roq-21}, \cite{Alf-Ham-Pat-Roq-23} and the references therein. Precisely, the fitness function admits a unique maximum and decreases quadratically away from it. Furthermore, while the phenotypic space changes through functions $A$ and $L$,  the location of the optimum $x=x_{opt}(t)$ also moves through the periodic function $\beta$, and the  intensity of selection varies through the periodic function $\alpha$. This connects to the issue of moving optima studied in \cite{Roq-Pat-Bon-Mar-20}, \cite{Lav-23}, among others, while the conjugate effects of periodic moving optimum and intensity of selection were analyzed in \cite{Fig-Mir-18, Fig-Mir-21}.

Hence model \eqref{eq} can be considered in both frameworks of ecology and evolutionary biology and, in the following, we may alternatively refer to habitat, spatial dispersal, growth or to phenotypic space, mutation, selection.

\medskip

In a series of very recent works, Allwright \cite{All-1, All-2, All-3} has addressed the issue  of reaction-diffusion problems posed on shifting and/or changing size domains, revealing sharp differences with the case of fixed domains. Among other things (such as the role of competition and/or the dimension) she proved that, when $\alpha\equiv 0$,  it may happen that the population survives although the habitat is always strictly smaller than the critical one, a phenomenon only possible in presence of a moving habitat.

Nevertheless, in these works, all individuals are assumed to be identical ($\alpha\equiv 0$) and selection is thus ignored. 
The purpose of this work is to take  into account this phenomenon ($\alpha > 0$) 
in the context of a moving phenotypic space and to investigate how the interplay of periodic selection and shifting/expanding domains affect population dynamics. To do so, we rely on both an analytical and a numerical approach. 

\medskip

The paper is organized as follows. In Section \ref{s:fixed-domain}, 
we consider the case of a fixed domain, and provide estimates on the underlying periodic principal eigenvalue, revealing contrasted outcomes depending on 
the fluctuations of the location of the optimum and of the intensity of selection. In Section \ref{s:sur-vs-ext}, we consider the moving habitat case and study the  {\it extinction vs. survival} issue for solutions to \eqref{eq}. To this end, we first transform the problem to one posed on  a fixed domain and construct appropriate sub- and super-solutions, whose long-time behavior can be analyzed using insights from  Section \ref{s:fixed-domain}. Finally, Section \ref{s:numerics} is dedicated to the numerical investigation of problem \eqref{eq}. To this end, we adopt the well-established space-time finite element method originally introduced in \cite{Hughes1982} and later extended, for example, in \cite{Moore_2018}, for the approximation of parabolic evolution problems on moving spatial domains. To improve stability, the method uses  a time-upwind test function, following the Streamline Upwind Petrov-Galerkin (SUPG) approach. We formulate this scheme to our specific setting, namely problem \eqref{eq}. Due to the lack of boundary regularity in our case, the underlying functional  framework is more intricate than the one considered in  \cite{Moore_2018}. We then establish the  well-posedness of the resulting discrete problem  and validate the method against an exact analytical solution in a simplified static setting. Finally, we  apply this approach  to perform simulations exploring different types of domain evolutions and revealing key insights into population survival/extinction dynamics, as discussed in detail in Section \ref{s:sur-vs-ext}. 

\section{The case of a fixed domain}\label{s:fixed-domain}

In this section, we  consider \eqref{eq} in the special case $A(t)=0$ (no shift of the domain), $L(t)=L$ (constant size of the domain), that is
\begin{equation}\label{eq-fix}
\begin{cases}
u_t=du_{xx}+\left(r-\frac{\alpha(t)}{2}\left(x-\beta(t) L\right)^2\right)u,  & \quad t>0,\, 0<x<L,\vspace{5pt}\\
u(t,0)=u(t,L)=0, & \quad t>0,
\end{cases}
\end{equation}
where we recall that both $\alpha$ and $\beta$ are Hölder continuous and $T$-periodic for some $T>0$.

We denote $\lambda=\lambda(\alpha,\beta)$ the principal eigenvalue, $\varphi=\varphi(t,x)$ the principal eigenfunction solving the periodic parabolic eigenproblem
\begin{equation}\label{pb-spectral}
\begin{cases}
\varphi _t -d\varphi_{xx}-\left(r-\frac{\alpha(t)}{2}\left(x-\beta(t)L\right)^2\right)\varphi=\lambda \varphi,  & \quad t\in \R, \, 0<x<L,\vspace{5pt}\\
\varphi(t,0)=\varphi(t,L)=0, &\quad t\in \R, \vspace{5pt}\\
\varphi>0, & \quad  t\in \R, \, 0<x<L, \vspace{5pt} \\
\varphi(t,x)=\varphi(t+T,x), & \quad  t\in \R, \, 0<x<L,
\end{cases}
\end{equation}
where $\varphi$ is normalized  by
\begin{equation}
    \label{normalization}
\Vert \varphi \Vert _{L^\infty(\R \times (0,L))}=1.
\end{equation}
The existence, uniqueness of such a principal eigenpair and the regularity of $\varphi$ are classical, see Theorem \ref{th:Cas-Laz-82} quoted from \cite{Cas-Laz-82}. For further details on such eigenelements, see \cite{book-Hes-91}, \cite{Can-Cos-96}, \cite{Hut-She-Vic-01}, \cite{Nad-09}, \cite{Liu-Lou-Pen-Zho-19}, \cite{Liu-Lou-Pen-Zho-21} and the references therein.

In the sequel, for a $T$-periodic function $t\mapsto R(t)$, we denote $\langle R\rangle$ its mean value that is
$$
\langle R \rangle:=\frac 1 T \int_0^T R(s)\,ds.
$$

It is well known that the sign of the principal eigenvalue of \eqref{pb-spectral} decides between survival  (when $\lambda\leq 0$) from extinction (when $\lambda>0$) in problem \eqref{eq-fix}. Hence the following estimate for $\lambda$ is crucial and is the main result of this section.

\begin{theorem}[Bounds for the eigenvalue]\label{th:bounds-eigenvalue} Let $L>0$. Let $t\mapsto \alpha(t)>0$, $t\mapsto \beta(t)$, be both Hölder continuous and $T$-periodic. Then,
 the principal eigenvalue $\lambda$ defined in \eqref{pb-spectral} enjoys the bounds
\begin{equation}\label{vp_estimate_final}
\frac{d\pi^2}{L^2}-r+\langle R^- \rangle  < \lambda < \frac{d\pi^2}{L^2}-r+L^2\frac{\langle \alpha \rangle}{2}\left(\frac{2\pi^2-3}{6\pi^2}+\frac{\langle\alpha\beta^2\rangle}{\langle\alpha\rangle}-\frac{\langle\alpha\beta\rangle}{\langle\alpha\rangle} \right),
\end{equation}
where
\begin{equation}\label{def:R-moins}
R^-(t):= \min_{0 \leq x \leq L}\left(\frac{\alpha(t)}{2}\left(x-\beta(t)L\right)^2\right).
\end{equation}
\end{theorem} 

\subsection{Bounds for the eigenvalue, proof of Theorem \ref{th:bounds-eigenvalue}}\label{ss:eigenvalue}

We begin with the following classical result: if the growth term is independent of $x$, the principal eigenelements can be explicitly determined. More precisely, the following statement holds.

\begin{lemma}\label{lem:no-x} If $t\mapsto R(t)$ is $T$-periodic and Hölder continuous, the principal eigenvalue corresponding to problem $u_t=du_{xx}+R(t)u$ (with zero Dirichlet boundary conditions as above) is nothing else than
$$
\frac{d\pi ^2}{L^2}-\langle R\rangle.
$$
\end{lemma}

\begin{proof} Denote $\phi(x):=\sin\left( \frac \pi L x\right)$ solving $-d\phi''=\frac{d\pi^2}{L^2}\phi$, $\phi(0)=\phi(L)=0$, $\phi>0$ on $(0,L)$. Plugging the ansatz $\phi(x)f(t)$ into the eigenvalue problem we are left to
$$
f'=\left(\lambda-\frac{d\pi^2}{L^2}+R(t)\right)f,
$$ 
whose nontrivial solutions are $T$-periodic if and only if $\lambda-\frac{d\pi^2}{L^2}+R(t)$ has zero mean, which gives the result.\end{proof}

Next, the principal eigenvalue is decreasing with respect to the growth term, see  \cite[Lemma 15.5]{book-Hes-91}. Precisely, the following holds. 

\begin{lemma}\label{lem:comp} For $i=1,2$, let $R_i\in \mathcal C^{\frac \nu 2, \nu}(\R \times [0,L])$ be $T$-periodic in time. Denote $\lambda_i$ the principal eigenvalue corresponding to problem $u_t=du_{xx}+R_i(t,x)u$ (with zero Dirichlet boundary conditions as above). Then
$$
R_1\leq R_2, R_1\not \equiv R_2 \Longrightarrow \lambda_2 <\lambda _1.
$$
\end{lemma}

From Lemma \ref{lem:no-x} and Lemma \ref{lem:comp}, we immediately infer the following bounds on the principal periodic parabolic eigenvalue. 

\begin{proposition}[First Bounds for the eigenvalue]\label{prop-estimate-vp}
The principal eigenvalue $\lambda$ defined in \eqref{pb-spectral} satisfies
\begin{equation} \label{estimate-vp}
\frac{d\pi^2}{L^2}-r+\langle R^- \rangle  < \lambda < \frac{d\pi^2}{L^2}-r+\langle R^+ \rangle,
\end{equation}
where
\begin{equation}\label{def:R}
R^-(t):= \min_{0 \leq x \leq L}\left(\frac{\alpha(t)}{2}\left(x-\beta(t)L\right)^2\right), \quad R^+(t):=\max_{0 \leq x \leq L}\left(\frac{\alpha(t)}{2}\left(x-\beta(t)L\right)^2\right).
\end{equation}
\end{proposition} 

The upper bound in \eqref{estimate-vp} is obtained by first maximizing the reaction term in space and, next, using the mean value in time thanks to the above lemmas. It turns out that taking the mean value in time first, which we do below, provides a better upper estimate.

 In \cite{Hut-She-Vic-01} Hutson, Shen and Vickers  proved that the principal eigenvalue of a periodic parabolic problem in the form of \eqref{pb-spectral} is  smaller than that of the corresponding elliptic problem obtained by taking the mean value of the growth rate. In other words, populations are more likely to persist in a fluctuating environment than in one with a constant averaged growth rate. Specifically, in our context \cite[Theorem 2.1]{Hut-She-Vic-01} establishes that $\lambda<\hat \lambda$, where $\hat \lambda$ is the  principal eigenvalue of the elliptic problem derived by rate averaging: 
\begin{equation}\label{pb-spectral_mean}
\begin{cases}
-d\varPsi_{xx}-\langle r-\frac{\alpha(t)}{2}\left(x-\beta(t)L\right)^2\rangle\varPsi=\hat{\lambda} \varPsi,  & \quad 0<x<L,\vspace{5pt}\\
\varPsi(0)=\varPsi(L)=0,\vspace{5pt} \\
\varPsi>0, & \quad   0<x<L.
\end{cases}
\end{equation}
Straightforward computations yield the more convenient equivalent problem
\begin{equation}\label{pb-spectral_mean2}
\begin{cases}
-d\varPsi_{xx}+ \frac{\langle \alpha\rangle}{2}\left(x-\frac{\langle\alpha\beta\rangle}{\langle\alpha\rangle} L\right)^2\varPsi=\hat\mu \varPsi,  & \quad  0<x<L,\vspace{5pt}\\
\varPsi(0)=\varPsi(L)=0, &\quad  \vspace{5pt}\\
\varPsi>0, & \quad   0<x<L, 
\end{cases}
\end{equation}
where
\begin{equation}\label{def_mu}
\hat\mu=\hat{\lambda}+r+\frac{L^2}{2}\left(\frac{\langle\alpha\beta\rangle^2}{\langle\alpha\rangle}-\langle \alpha\beta^2\rangle\right).
\end{equation}
Since this new eigenvalue problem is associated with a self-adjoint operator, we are equipped with  the variational formulation
\begin{equation}\label{Rayleigh}
\hat\mu=\inf_{u\in H^1_0(0,L), \lVert u \rVert_{L^2}=1} Q_d(u), \quad Q_d(u):=d\int_0^L  u_x^2 \,dx+  \frac{\langle\alpha\rangle}{2}\int_0^L \left(x-\frac{\langle\alpha\beta\rangle}{\langle\alpha\rangle} L\right)^2 u^2 \,dx.
\end{equation} 
From this characterization, we can draw out some properties of $\hat\mu$. 

\begin{lemma} Let $L>0$. Let $t\mapsto \alpha(t)>0$, $t\mapsto \beta(t)$, be both Hölder continuous and $T$-periodic. Then the  principal eigenvalue $\hat \mu$ enjoys the following properties. 
\begin{enumerate}
\item[(i)] The function $d\mapsto\hat\mu=\hat\mu(d)$ is increasing and concave on $(0,+\infty)$, and
\begin{equation*}
\lim\limits_{d \to 0} \hat\mu=L^2\frac{\langle \alpha \rangle}{2}\min_{0\leq x \leq 1 }\left(x-\frac{\langle\alpha\beta\rangle}{\langle\alpha\rangle}\right)^2 \qquad \text{and}\qquad \lim\limits_{d \to+\infty} \hat\mu=+\infty.
\end{equation*}
\item[(ii)]In addition, for any $d>0$, we have 
\begin{equation}\label{estimate_mu}
 \frac{d\pi^2}{L^2}< \hat\mu < \frac{d\pi^2}{L^2}+L^2\frac{\langle \alpha \rangle}{2}\left(\frac{2\pi^2-3}{6\pi^2} +\frac{\langle\alpha\beta\rangle}{\langle\alpha\rangle}\left(\frac{\langle\alpha\beta\rangle}{\langle\alpha\rangle}-1\right)\right).
\end{equation}
\end{enumerate}
\end{lemma}

\begin{proof} In $(i)$, the fact that the function $d\mapsto \hat \mu(d)$ is increasing follows from the Rayleigh formula \eqref{Rayleigh} and  is classical, see e.g. \cite{Hut-95}. Also, the map $d\mapsto Q_d(u)$ being linear (hence concave) in $(0,+\infty)$ for each $u\in H_0^1(0,L)$, the map $d\mapsto \hat \mu (d)$ is concave in $(0,+\infty)$ and therefore continuous. Also, from \eqref{Rayleigh} we get:
\begin{equation*}
\hat\mu> \inf_{u\in H^1_0(0,1)\, , \, \lVert u \rVert_{L^2}=1} d\int_0^L  u_x^2 \,dx\,=\,\frac{d\pi^2}{L^2}.
\end{equation*} 

Next, for the behavior as $d\to 0$, we may use similar arguments as those in  \cite{Alf-Ver-18}. We consider $\Phi$ a smooth and nonnegative function on $\R$, compactly supported in $(-1,1)
$, and normalized by $\lVert \Phi \rVert_{L^2(\R)}=1$. Let $\sigma\in(0,1)$ be given. Choose $d>0$ small enough so that
$$
\frac{d^{\frac{1}{4}}}{L}\leq\sigma\leq 1- \frac{d^{\frac{1}{4}}}{L}.
$$
Consequently, the function  $\Phi_d(x):=\frac{1}{d^{\frac{1}{8}}}\Phi\left(\frac{x-\sigma L}{d^{\frac{1}{4}}}\right)$ belongs to $H_0^1(0,L)$ and remains $L^2$ normalized. From \eqref{Rayleigh}, we thus get
\begin{equation*}
L^2 \frac{\langle\alpha\rangle}{2}\min_{0\leq x \leq 1 }\left(x-\frac{\langle\alpha\beta\rangle}{\langle\alpha\rangle}\right)^2\leq\hat\mu \leq d\int_0^L  (\Phi_d^{\prime}(x))^2 \,dx+\frac{\langle\alpha\rangle}{2}\int_0^L \left(x-\frac{\langle\alpha\beta\rangle}{\langle\alpha\rangle} L\right)^2 \Phi_d^2(x) \,dx.
\end{equation*}  
Then, by expressing $\Phi_d$ and by using the change of variable $y=\frac{x-\sigma L}{d^{\frac{1}{4}}}$,
\begin{equation*}
\hat\mu \leq \sqrt{d}\lVert  \Phi^{\prime} \rVert_{L^2(\R)}+ \frac{\langle \alpha\rangle}{2}\int_{\R} \left(yd^{\frac{1}{4}}+\left(\sigma-\frac{\langle\alpha\beta\rangle}{\langle\alpha\rangle}\right) L\right)^2 \Phi^2(y) \,dy.
\end{equation*} 
We let $d\to 0$ and deduce from the dominated convergence theorem  that
\begin{equation*}
L^2 \frac{\langle\alpha\rangle}{2} \min_{0\leq x \leq 1 }\left(x-\frac{\langle\alpha\beta\rangle}{\langle\alpha\rangle}\right)^2\leq \lim\limits_{d \to 0} \hat\mu\leq L^2 \frac{\langle\alpha\rangle}{2} \left(\sigma-\frac{\langle\alpha\beta\rangle}{\langle\alpha\rangle}\right)^2.
\end{equation*}
This being true for any $0<\sigma <1$ we get $\lim\limits_{d \to 0} \hat\mu=L^2 \frac{\langle\alpha\rangle}{2} \min_{0\leq x \leq 1 }\left(x-\frac{\langle\alpha\beta\rangle}{\langle\alpha\rangle}\right)^2$.

Last the upper estimate in \eqref{estimate_mu} comes from testing $Q_d$ with the normalized eigenfunction of the Laplacian Dirichlet, namely $u(x)=\sqrt{\frac{2}{L}}\sin\left( \frac{\pi}{L}x\right)$, and very straightforward computations.
\end{proof}

As a result,  \eqref{estimate_mu} and  \eqref{def_mu} imply the improved upper estimate
\begin{equation}\label{truc}
\lambda < \frac{d\pi^2}{L^2}-r+L^2\frac{\langle \alpha \rangle}{2}\left(\frac{2\pi^2-3}{6\pi^2}+\frac{\langle\alpha\beta^2\rangle}{\langle\alpha\rangle}-\frac{\langle\alpha\beta\rangle}{\langle\alpha\rangle} \right),
\end{equation}
which, combined with Proposition \ref{prop-estimate-vp} completes the proof of 
\eqref{vp_estimate_final} and of Theorem \ref{th:bounds-eigenvalue}. \qed

\begin{remark} We claimed that the upper bound in \eqref{truc} is better than the one in \eqref{estimate-vp}. To see this, observe that the maximum defining $R^+(t)$ in \eqref{def:R} is reached at $x=0$ when $\beta(t)\geq \frac 12$ and is equal to $\frac{\alpha(t)}2 \beta^2(t)L^2$, and is reached at $x=L$ when $\beta(t)<\frac 12$ and is equal to $\frac{\alpha(t)}{2}(1-\beta(t))^2L^2$. In other words
\begin{equation}\label{def:R+2}
R^+(t)=\frac{\alpha(t)}{2}\left(\frac 12+\vert \frac 12-\beta(t)\vert\right)^2L^2,
\end{equation}
and what we have to check is 
\begin{equation}\label{goal}
\frac{2\pi^2-3}{6\pi^2} \langle \alpha \rangle +\langle\alpha\beta^2\rangle-\langle\alpha\beta\rangle\leq \left\langle \alpha \left(\frac 12+\vert \frac 12-\beta\vert\right)^2 \right\rangle .
\end{equation}
Then, defining
$A:=\{ t\in (0,T):\beta(t)>\frac{1}{2} \}$,  $B:=\{ t\in (0,T): \beta(t)\leq \frac{1}{2} \}$, and decomposing all integrals over $A$ and $B$, it is straightforward to check that \eqref{goal} is recast
\begin{equation*}
\int_{A}\alpha(t)\left(\frac{2\pi^2-3}{6\pi^2}-\beta(t)\right)\, dt \leq \int_{B} \alpha(t)\left(1-\beta(t)-\frac{2\pi^2-3}{6\pi^2}\right)\, dt,
\end{equation*}
which is obviously true since the left hand side is negative while the right hand side is positive.
\end{remark}

\subsection{Sign of the eigenvalue}\label{ss:sign}

First, in absence of selection ($\alpha\equiv  0$), it is well known that too small domains lead to extinction and, by comparison,  the same holds for \eqref{eq-fix}. Precisely, if $0<L\leq \sqrt{\frac{d\pi^2}{r}}$ it follows from \eqref{vp_estimate_final} that $\lambda>0$.

Furthermore, due to the effects of selection, infinitely expanding the domain size may also result in extinction,  contrasting with the behavior observed in the homogeneous case.

\begin{corollary}[Extinction criterion]\label{cor:ext} Assume that, for all $t \geq 0$, $\beta(t)\notin (0,1)$. Assume either that 
\begin{equation*}
r^2\leq 2d\pi^2 \left \langle \alpha\left(\frac 12 -\vert \frac 1 2-\beta\vert\right)^2\right\rangle=:\delta^-,
\end{equation*}
or that
\begin{equation*}
r^2> \delta^-, \text{ and } \left(
L\leq \sqrt{2d\pi^2\frac{r-\sqrt{r^2-\delta^-}}{\delta^-}}\quad \text{or}\quad L\geq \sqrt{2d\pi^2\frac{r+\sqrt{r^2-\delta^-}}{\delta^-}}\right).
\end{equation*}
Then $\lambda>0$.
\end{corollary}

\begin{proof}
The minimum defining $R^-(t)$ in \eqref{def:R-moins} is reached at $x=0$ when $\beta(t)\leq 0$ and is equal to $\frac{\alpha(t)}2 \beta^2(t)L^2$, and is reached at $x=L$ when $\beta(t)\geq 1$ and is equal to $\frac{\alpha(t)}{2}(1-\beta(t))^2L^2$. In other words, when $\beta(t)\not\in(0,1)$, 
\begin{equation}\label{def:R-2}
R^-(t)=\frac{\alpha(t)}{2}\left(\frac 12-\vert \frac 12-\beta(t)\vert\right)^2L^2,
\end{equation}
and it follows from \eqref{vp_estimate_final} that
$$
\lambda >\frac{d\pi^2}{L^2}-r+\frac{L^2}{2}\frac{\delta^-}{2d\pi^2} ,
$$
from which the result is a straightforward consequence.
\end{proof}

Actually, the reason for the above counter-intuitive phenomenon  is that, as $L$ increases, the optimum located at $x=\beta(t)L$ is getting further away from the domain when $\beta(t)\not\in (0,1)$. 

Next, regardless of the location of the optimum, provided that the growth rate is sufficiently large, there exists an interval of domain sizes, which is enlarging with respect to $r$, that grants survival. Precisely, the following holds.

\begin{corollary}[Survival criterion]\label{cor:sur} Assume that the growth rate is such that
 $$
r^2\geq 2d\pi^2 \left(\frac{2\pi^2-3}{6\pi^2}\langle \alpha \rangle+\langle\alpha\beta^2\rangle-\langle\alpha\beta\rangle \right)=:\delta^+,
$$
and that the interval length is such that
\begin{equation*}
\sqrt{2d\pi^2\frac{r-\sqrt{r^2-\delta^+}}{\delta^+}}\leq L\leq \sqrt{2d\pi^2\frac{r+\sqrt{r^2-\delta^+}}{\delta^+}}. 
\end{equation*}
Then $\lambda<0$.
\end{corollary}

\begin{proof}
It follows from \eqref{vp_estimate_final} that
$$
\lambda < \frac{d\pi^2}{L^2}-r+\frac{L^2}{2}\frac{\delta^+}{2d\pi^2},
$$
from which the result is a straightforward consequence.
\end{proof}

\begin{example}[Optimum outside the domain]\label{ex:outside}
When $\beta(t)\notin (0,1)$, meaning that the optimum always lies outside the domain $(0,L)$, we can combine Corollaries \ref{cor:ext} and \ref{cor:sur}  to earn insight on parameters regions of extinction and survival, see Figure \ref{fig:parabola}, where we have selected  $\alpha(t)=4$ and $\beta(t)=1.5$.
\begin{figure}[h!]
   \centering
  \includegraphics[scale=0.23]{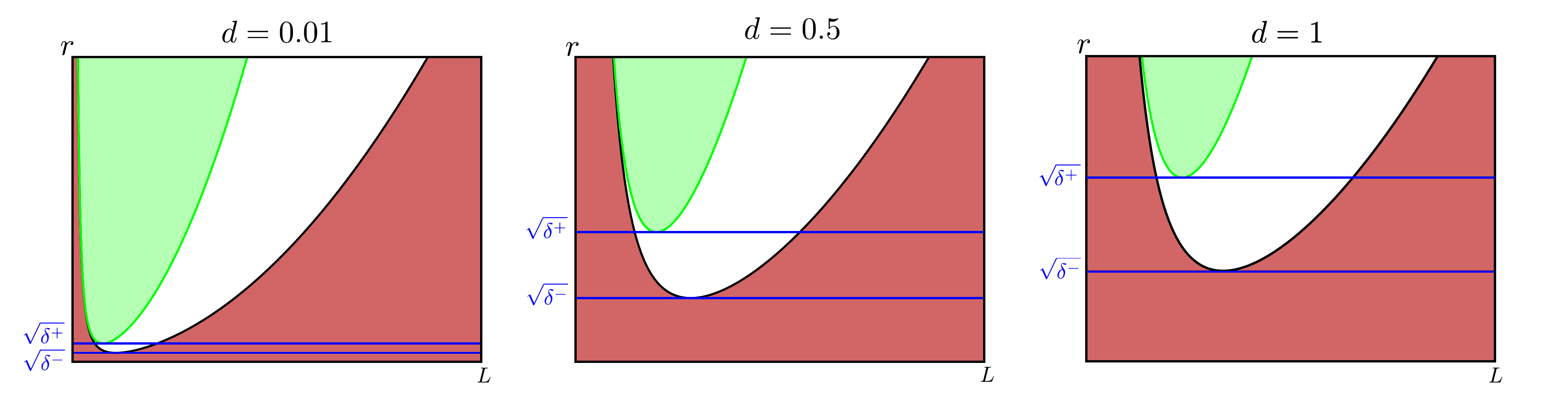}       \caption{\textbf{Sign of the eigenvalue as a function of $(L,r)$.}  The parameters are as described in Example \ref{ex:outside}. In green, the survival zone ($\lambda<0$). In red, the extinction zone ($\lambda>0$). In white, the zone where our estimates are not enough to conclude. 
  }
   \label{fig:parabola}
\end{figure}
Note that in this case, for a fixed set of parameters $d$, $\alpha$, $\beta$, and $r>\sqrt{\delta^+}$, the principal eigenvalue $\lambda$ is not monotone with respect to $L$. Indeed, as already emphasized above, enlarging the domain may decrease the chances of survival. Also, considering for instance  the case $\beta (t)\geq 1$, one has
 $$
 \delta^+-\delta^-=2d\pi^2 \left(\frac{2\pi^2-3}{6\pi^2}\langle \alpha\rangle -\langle \alpha (1-\beta) \rangle  \right),
 $$
meaning that, the smaller the diffusion coefficient $d$, the smaller
 the region of parameters $(L,r)$ for which we cannot conclude on the sign of the eigenvalue, see the white regions in Figure  \ref{fig:parabola}.

However, to gain insight on these uncertain (white) regions of parameters, we may run simulations of \eqref{eq-fix}. For instance, we consider the same case as in Figure \ref{fig:parabola} with  $d=1$, $\alpha (t)=4$, $\beta(t)=1.5$. Then, for a given $r$, we consider the $L^2$ norm of the solution $u(T=2,\cdot)$ starting from the initial datum $u_0(x)=\sin(\frac{\pi}{L}x)$ as a function of the length $L$ of the domain. For $r=15>\sqrt{\delta^+}$, we observe (left panel of Figure \ref{fig:conjecture}) three successive $L$-ranges of extinction-survival-extinction. On the other hand when $\sqrt{\delta^-}<r=5<\sqrt{\delta^+}$, we observe (right panel of Figure \ref{fig:conjecture}) systematic extinction. We conjecture that, at least in the case where $\alpha$ and $\beta$ are constant, there is a threshold value $r^*>0$ that separates \lq\lq systematic extinction'' ($0<r<r^*$) from \lq\lq extinction-survival-extinction depending on $L$'' ($r> r^*)$. However, non-constant $\alpha$ and $\beta$ may lead to  more complex scenarios. 
\begin{figure}[h!]
   \centering
  \includegraphics[scale=0.4]{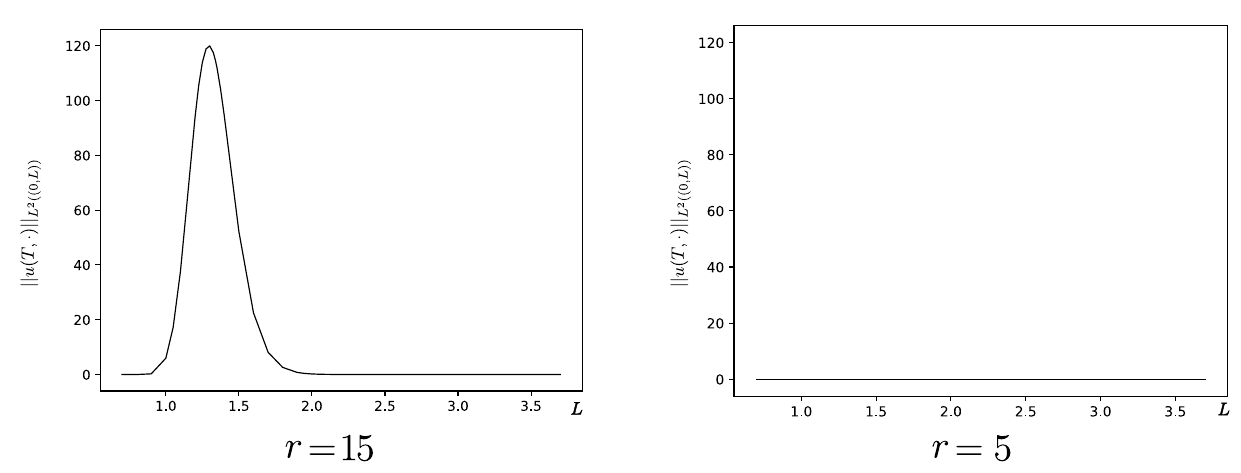}       \caption{\textbf{Evolution of the $L^2$ norm of the solution at final time $T=2$ as a function of the length $L$ of the domain.}  The parameters and the initial datum  are  as described in Example \ref{ex:outside}.
  }
   \label{fig:conjecture}
\end{figure}
\end{example}

Numerical simulations illustrating the  dynamical evolution of problem \eqref{eq-fix} will be presented  in  Section \ref{s:numerics}.

\subsection{To fluctuate, or not to fluctuate}\label{ss:fluctuate}

Here we aim at comparing the chances of survival between a constant fitness and a fluctuating fitness, see \cite{Fig-Mir-18, Fig-Mir-21} and the references therein for related issues.

First, we look for a situation where fluctuations increase the chances of survival. Let us thus consider the  constant fitness case, namely
\begin{equation*}\label{fixed_reaction}
\alpha_1(t)=\alpha >0, \quad  \beta_1(t)=\beta>1,
\end{equation*}
and denote $\lambda_1$ the associated eigenvalue. From Theorem \ref{th:bounds-eigenvalue}, we have
\begin{equation}\label{bounds-lambda1}
\frac{\alpha}{2}L^2 \left(1-\beta\right)^2< \lambda_1 -\frac{d\pi^2}{L^2}+r< \frac{\alpha}{2}L^2\left(\frac{2\pi^2-3}{6\pi^2}+\beta^2-\beta\right).
\end{equation}
Next, we  consider a fluctuating situation described by some $\alpha_2(t)$ with mean $\alpha$, and some $\beta_2(t)$ with mean $\beta$. If the position of the optimum, given by $\beta_2(t)$, is constant, we already know from \cite{Hut-She-Vic-01}, see \eqref{pb-spectral_mean}, that the fluctuation of the strength of the selection $\alpha_2(t)$ increases the chances of survival. However, if $\beta_2(t)$ is non constant, the quadratic term makes the outcome more tricky. We thus consider
\begin{equation}\label{fluc-aide}
\alpha_2(t)=\alpha+a\sin(\omega t), \quad \beta_2(t)=\beta-b\sin(\omega t),
\end{equation}
for $\omega=\frac{2\pi}{T}>0$, $0<a<\alpha$, $b\in \R$, and  denote $\lambda_2$ the associated  eigenvalue. It follows from  $\langle \alpha_2\beta _2\rangle=\alpha \beta -\frac{ab}{2}$, $\langle \alpha_2 \beta_2^2\rangle=\alpha \beta^2+\frac{\alpha b^2}{2}-a\beta b$ and Theorem \ref{th:bounds-eigenvalue} that
$$\lambda_2 -\frac{d\pi^2}{L^2}+r<\frac{\alpha }{2}L^2 \left(\frac{2\pi^2-3}{6\pi^2}+\beta^2-\beta-\frac{a}{\alpha}b(\beta-\frac12)+\frac{b^2}{2}\right).
$$
Hence, for $\lambda_2<\lambda_1$ to hold it is enough to have
\begin{equation*}
b^2-\frac{a}{\alpha}\left(2\beta-1\right)b +2\left(\frac{2\pi^2-3}{6\pi^2}-1+\beta\right)<0.
\end{equation*}
Some elementary computations reveal that this inequality is true as soon as 
\begin{equation}\label{condition_Beta_tilde+}
\beta > \tilde{\beta}:=\frac12+\frac{\alpha^2}{a^2}\left(1+\sqrt{\frac{a^2}{\alpha^2}\left(\frac{2\pi^2-3}{3\pi^2}-1\right)+1}\right),
\end{equation}
and 
\begin{equation}\label{condition_b1<b<b2+}
b_1<b<b_2,
\end{equation}
with 
\begin{equation}\label{def_b1}
b_1:=\frac{a}{\alpha}\left(\beta-\frac12\right)-\sqrt{\frac{a^2}{\alpha^2}\left(\beta-\frac12\right)^2-2\left(\beta+\frac{2\pi^2-3}{6\pi^2}-1\right)}\; >0,
\end{equation}
and
\begin{equation}\label{def_b2}
b_2:=\frac{a}{\alpha}\left(\beta-\frac12\right)+\sqrt{\frac{a^2}{\alpha^2}\left(\beta-\frac12\right)^2-2\left(\beta+\frac{2\pi^2-3}{6\pi^2}-1\right)}\; >0.
\end{equation}
Observe that  $b$ has to be positive, meaning that a key element to improve the chances of survival is the phase opposition between $\alpha_2$ and $\beta_2$, which can be interpreted as a balancing effect between the strength of the selection and the position of the optimum growth. In other words, when the selection becomes stronger the optimum needs to get closer to the domain, while when the optimum goes far from the domain the selection needs to becomes weaker to compensate. We retain the following.

\begin{example}[Fluctuations may help]\label{ex:fluc-help}
 Let $0<a <\alpha$. Let $\beta>\tilde \beta>1$, where $\tilde \beta$ is defined in  \eqref{condition_Beta_tilde+}. Let $b\in (b_1,b_2)$,  where $b_1$, $b_2$ are defined in \eqref{def_b1} and \eqref{def_b2}. Then $\lambda_2<\lambda _1$. 
\end{example} 

Next we look for a situation where fluctuations decrease the chances of survival. The constant fitness case is as above, in particular we have \eqref{bounds-lambda1}. According to the preceding part, we expect that the fluctuations in the position of the optimum and the strength of selection should be {\it in phase}. For the fluctuating case, rather than \eqref{fluc-aide}, we thus use the convention
\begin{equation*}
\alpha_2(t)=\alpha+a\sin(\omega t), \quad \beta_2(t)=\beta+b\sin(\omega t),
\end{equation*} 
for  $\omega=\frac{2\pi}{T}>0$, $0<a<\alpha$, $b\in\R$, and denote $\lambda_2$ the associated  eigenvalue.  Under the additional assumption 
\begin{equation}\label{assumption_beta-b}
\beta-b>1,
\end{equation}
(to be checked {\it a posteriori}),  it follows from Theorem \ref{th:bounds-eigenvalue} and  a straightforward computation that
\begin{equation*}
\langle R^-_2\rangle = \frac{L^2}{2}\left(\alpha(\beta-1)^2+\frac{\alpha b^2+2ab(\beta-1)}{2}\right)<\lambda_2.
\end{equation*}
Hence, for $\lambda_1<\lambda_2$ to hold  it is enough to have
\begin{equation*}
b^2+2\frac{a}{\alpha}(\beta-1) b +2\left(1-\beta-\frac{2\pi^2-3}{6\pi^2}\right)>0.
\end{equation*}
This is obviously true as soon as 
\begin{equation}\label{condition_b3}
b>b_3:= -\frac{a}{\alpha}(\beta-1)+\sqrt{\frac{a^2}{\alpha^2}(\beta-1)^2+2\left(\beta+\frac{2\pi^2-3}{6\pi^2}-1\right)}>0.
\end{equation}
Accordingly with the above remark, $b>0$ is mandatory, meaning that $\alpha_2(t)$ and $\beta_2(t)$ are in phase. Last, we also need $b<\beta-1$, which provides a condition on $\beta$ for the set of possible $b$'s not to be empty. One can check that the following holds.

\begin{example}[Fluctuations may hurt]\label{ex:fluc-hurt}
Let $0<a<\alpha$. Let $\beta>1$ be large enough so that
$$
\beta>1+\frac{1+\sqrt{1+2\frac{2\pi^2-3}{6\pi^2}\left(1+2\frac{a}{\alpha}\right)}}{1+2\frac{a}{\alpha}}.
$$
Let $b\in (b_3,\beta - 1 )$, where $b_3$ is defined in \eqref{condition_b3}. Then $\lambda_1<\lambda_2$. 
\end{example} 

\section{Survival vs. extinction}\label{s:sur-vs-ext}

In this section, we consider the moving habitat case as stated in \eqref{eq}-\eqref{x-opt} and aim at understanding the long time behavior of its solutions $u=u(t,x)$. We should emphasize that the construction of sub- and supersolutions in subsections \ref{ss:switch} and \ref{ss:subsuper} is inspired by that  performed in \cite{All-3, All-1}. However, in the problems considered there,  $\alpha \equiv 0$ was assumed, i.e. the effect of  selection  was ignored. 

\subsection{Switching  to a fixed domain}\label{ss:switch}

We first change the spatial/phenotypic variable to switch to an equation on a fixed domain. For $L_0>0$, we write
\begin{equation}\label{u-v}
u(t,x)=v(t,y),  \quad y:=\frac{x-A(t)}{L(t)}L_0,
\end{equation}
and reach 
\begin{equation}\label{eq-v}
\begin{cases}
v_t=\frac{dL^2_0}{L^2(t)}\, v_{yy}+\frac{\dot{A}(t)L_0+y\dot{L}(t)}{L(t)}\, v_{y}+\left(r-\frac{\alpha(t) L^2(t)}{2L^2_0}\left(y-\beta(t) L_0\right)^2\right)v,  & \quad  t>0,\, 0<y<L_0,\vspace{5pt}\\
v(t,0)=v(t,L_0)=0, & \quad  t>0,
\end{cases}
\end{equation}
which is, obviously, a reaction-advection-diffusion equation with time and space dependent coefficients. 

Next, to suppress the advection term, we change the unknown function through  
\begin{equation}\label{v-w}
w(t,y):=v(t,y)\left(\frac{L(t)}{L_0}\right)^{1/2}\exp\left(-rt+\int_{0}^{t}\frac{\dot{A}(s)^2}{4d}ds\right)\exp\left(\frac{y^2\dot{L}(t)L(t)}{4dL^2_0}+\frac{y\dot{A}(t)L(t)}{2dL_0}\right),
\end{equation}
 and reach
\begin{equation}\label{eq-w}
\begin{cases}
w_t=\frac{dL^2_0}{L(t)^2}w_{yy}+\left(\frac{\ddot{L}(t)L(t)}{4dL^2_0}\,y^2+\frac{\ddot{A}(t)L(t)}{2dL_0}\, y-\frac{\alpha(t) L^2(t)}{2L^2_0}\left(y-\beta(t) L_0 \right)^2\right)w,  & \quad  t>0,\, 0<y<L_0,\vspace{5pt}\\
w(t,0)=w(t,L_0)=0, & \quad  t>0.
\end{cases}
\end{equation}

\subsection{Construction of sub- and supersolutions}\label{ss:subsuper}
 
Now, to  get closer from an equation having the form of \eqref{eq-fix}, we define 
\begin{equation}\label{Q+}
 \overline{Q}(t):=\max_{0\leq z\leq1}\left(\frac{\ddot{L}(t)L(t)}{4d}z^{2}+\frac{\ddot{A}(t)L(t)}{2d}z\right),
\end{equation}
\begin{equation}\label{Q-}
\underline{Q}(t):=\min_{0\leq z\leq1}\left(\frac{\ddot{L}(t)L(t)}{4d}z^{2}+\frac{\ddot{A}(t)L(t)}{2d}z\right),
\end{equation}
so that $w=w(t,y)$ solving \eqref{eq-w} is a subsolution for the problem
\begin{equation}\label{eq_with_Q+}
\begin{cases}
\overline{w}_t=\frac{dL^2_0}{L(t)^2}\overline{w}_{yy}+\left(\overline{Q}(t)-\frac{\alpha(t) L^2(t)}{2L^2_0}\left(y-\beta(t) L_0 \right)^2\right)\overline{w},  & \quad  t>0,\, 0<y<L_0,\vspace{5pt}\\
\overline{w}(t,0)=\overline{w}(t,L_0)=0, & \quad  t>0,
\end{cases}
\end{equation}
and a supersolution for the problem
\begin{equation}\label{eq_with_Q-}
\begin{cases}
\underline{w}_t=\frac{dL^2_0}{L(t)^2}\underline{w}_{yy}+\left(\underline{Q}(t)-\frac{\alpha(t) L^2(t)}{2L^2_0}\left(y-\beta(t) L_0 \right)^2\right)\underline{w},  & \quad  t>0,\, 0<y<L_0,\vspace{5pt}\\
\underline{w}(t,0)=\underline{w}(t,L_0)=0, & \quad  t>0.
\end{cases}
\end{equation}

We look for a  supersolution to \eqref{eq_with_Q+} in the form 
\begin{equation}\label{claim_sursol}
\omega^+(t,y):=\varphi\left(\int_0^t \frac{L_0^2}{L^2(s)}ds,y\right)\exp\left(f(t)+\int_0^t\overline{Q}(s)ds\right),
\end{equation}
with $f=f(t)$ to be selected.  Here, $\varphi=\varphi(\tau,y)$ denotes the $L^\infty$-normalized principal eigenfunction solving
\eqref{pb-spectral} on the interval $(0,L_0)$ and associated with the principal eigenvalue denoted $\lambda$ (estimated in  Section \ref{s:fixed-domain}), namely
\begin{equation}\label{pb-spectral-sur-0-L_0}
\begin{cases}
\varphi _\tau -d \varphi_{yy}-\left(r-\frac{\alpha(\tau)}{2}\left(y-\beta(\tau)L_0\right)^2\right)\varphi=\lambda \varphi,  & \quad \tau\in \R, \, 0<y<L_0,\vspace{5pt}\\
\varphi(\tau,0)=\varphi(\tau,L_0)=0, &\quad \tau\in \R, \vspace{5pt}\\
\varphi>0, & \quad  \tau\in \R, \, 0<y<L_0, \vspace{5pt} \\
\varphi(\tau,y)=\varphi(\tau+T,y), & \quad  \tau\in \R, \, 0<y<L_0.
\end{cases}
\end{equation}
Plugging the ansatz  \eqref{claim_sursol} into \eqref{eq_with_Q+}, it follows from straightforward computations that the choice
\begin{equation}\label{def-f}
f(t)=\int_0^t \left(-\frac{(\lambda+r)L_0^2}{L^2(s)}+\overline P(s)\right)  ds
\end{equation}
where
\begin{eqnarray}
\overline P(s):=\frac 12\max_{0\leq y\leq L_0} & \Bigg[\frac{L_0^2}{L^2(s)}\alpha\left(\int_0^s \frac{L_0^2}{L^2(T)}dT\right)\left(y-\beta\left(\int_0^s \frac{L_0^2}{L^2(T)}dT\right)L_0\right)^2\nonumber\\
&-\frac{L^2(s)}{L_0^2}\alpha(s)\left(y-\beta(s)L_0\right)^2\Bigg]\label{overline-P}
\end{eqnarray}
does make $\omega^+$ a supersolution to \eqref{eq_with_Q+}.

Similarly, the choice 
\begin{equation}\label{def-g}
g(t)=\int_0^t \left(-\frac{(\lambda+r)L_0^2}{L^2(s)}+\underline P (s) \right)  ds
\end{equation}
where
\begin{eqnarray}
\underline P(s):=\frac 12 \min_{0\leq y\leq L_0} & \Bigg[\frac{L_0^2}{L^2(s)}\alpha\left(\int_0^s \frac{L_0^2}{L^2(T)}dT\right)\left(y-\beta\left(\int_0^s \frac{L_0^2}{L^2(T)}dT\right)L_0\right)^2\nonumber\\
&-\frac{L^2(s)}{L_0^2}\alpha(s)\left(y-\beta(s)L_0\right)^2\Bigg]\label{underline-P}
\end{eqnarray}
makes
\begin{equation}\label{claim_soussol}
\omega^-(t,y):=\varphi\left(\int_0^t \frac{L_0^2}{L^2(T)}dT,y\right)\exp\left(g(t)+\int_0^t\underline{Q}(s)ds\right)
\end{equation}
a subsolution to \eqref{eq_with_Q-}.

Putting all together, we have the following.

\begin{theorem}[Bounds for the solution]\label{th:sub-supersol} Let $u=u(t,x)$ be the solution to \eqref{eq} starting from 
a nonnegative and nontrivial $u_0\in L^\infty(A(0),A(0)+L(0))$, or equivalently, for any $L_0>0$, $v=v(t,y)$ the solution to \eqref{eq-v} starting from $v_0(y)=u_0\left(\frac{L(0)}{L_0}y+A(0)\right)$, or equivalently $w=w(t,y)$ the solution to \eqref{eq-w}
 starting from $w_0(y)=v_0(y)\left(\frac{L(0)}{L_0}\right)^{1/2}\exp\left(\frac{\dot L (0)L(0)}{4dL_0^2}y^2+\frac{\dot A(0)L(0)}{2dL_0}y \right)$. Assume there are $0<a<b<+\infty$ such that
\begin{equation}\label{initial-order}
a\varphi\left(0,y\right)\leq w_0(y)\leq b \varphi\left(0,y\right), \quad 0< y< L_0.
\end{equation}

Then, for any $t>0$, any $0<y<L_0$, 
\begin{align}
&v(t,y)  \leq    b\,\varphi \left(\int_0^t \frac{L_0^2}{L^2(s)}ds,y\right) \left(\frac{L_0}{L(t)}\right)^{1/2}\nonumber\\
&\quad \times \exp\left( rt+\int_{0}^{t} \left( -\frac{\dot{A}^2(s)}{4d}-\frac{(\lambda+r)L_0^2}{L^2(s)}+\overline{P}(s)+\overline{Q}(s)\right)\,ds
-\frac{\dot{L}(t)L(t)}{4dL_0^2}y^{2}-\frac{\dot{A}(t)L(t)}{2dL_0}y\right),\label{sursol2}
\end{align}
and
\begin{align}
&v(t,y)  \geq  a\,\varphi\left(\int_0^t \frac{L_0^2}{L^2(s)}ds,y\right)\left(\frac{L_0}{L(t)}\right)^{1/2}\nonumber \\
&\quad  \times \exp\left( rt+\int_{0}^{t} \left( -\frac{\dot{A}^2(s)}{4d}-\frac{(\lambda+r)L_0^2}{L^2(s)}+\underline{P}(s)+\underline{Q}(s)\right)\,ds-\frac{\dot{L}(t)L(t)}{4dL_0^2}y^{2}-\frac{\dot{A}(t)L(t)}{2dL_0}y\right),\label{soussol2}
\end{align}
where $\overline Q$, $\underline{Q}$ are defined in \eqref{Q+}, \eqref{Q-}, while 
$\overline P$, $\underline P$ are defined in \eqref{overline-P}, \eqref{underline-P}.
\end{theorem}

\begin{proof} Since \eqref{initial-order} is nothing else than $a\omega^-(0,y)\leq w_0(y)\leq b \omega^+(0,y)$, we deduce from the above analysis and the comparison principle that 
$a\omega^-(t,y)\leq w(t,y)\leq b\omega^+(t,y)$ for any $t>0$, any $0<y<L_0$. Using the expressions \eqref{claim_sursol}, \eqref{claim_soussol}, and returning to $v$ via \eqref{v-w}, we  reach the conclusion.
\end{proof}

\begin{remark}\label{rem:fully} If $A$ and $L$ are, as $\alpha$ and $\beta$, $T$-periodic, some refinements are achievable. Indeed, if we further assume that $\ddot{A}$, $\ddot{L}$ are Hölder continuous, and 
\begin{equation}
2d\alpha(t) L(t)-\ddot{L}(t)>0, \quad \forall t\in \R, 
\end{equation}
we can rewrite the equation in \eqref{eq-w} as
 \begin{equation}\label{eq-w tilde}
w_t=\frac{dL^2_0}{L(t)^2}w_{yy}+\left({Q(t)-\frac{\tilde{\alpha}(t)}{2}\frac{L^2(t)}{L_0^2}} \left(y-\tilde{\beta}(t) L_0 \right)^2\right) w,   \quad  t>0,\, 0<y<L_0,
\end{equation}
where
\begin{equation*}
Q(t):=\frac{L(t)}{4d}\frac{\left(\ddot{A}(t)+2d\alpha(t)\beta(t)L(t)\right)^2}{2d\alpha(t)L(t)-\ddot{L}(t)}-\frac{\alpha(t)\beta^2(t) L^2(t)}{2},
\end{equation*}
and
\begin{equation*}
\tilde{\alpha}(t):=\frac{1}{2dL(t)}\left(2d\alpha(t)L(t)-\ddot{L}(t)\right),\quad 
\tilde{\beta}(t):=\frac{\ddot{A}(t)+2d\alpha (t)\beta (t)L(t)}{2d\alpha(t)L(t)-\ddot{L}(t)}.
\end{equation*}
Then, one can check that \eqref{sursol2} and \eqref{soussol2} are still valid after replacing both $\overline Q$ and $\underline Q$ by $Q$, while in the definitions \eqref{overline-P} and \eqref{underline-P} of $\overline P$ and $\underline P$, $\alpha$ and $\beta$ are replaced by $\tilde \alpha$ and $\tilde \beta$.
\end{remark}

\subsection{The shift effect}\label{ss:shift}

In this short subsection, we take advantage of Theorem \ref{th:sub-supersol} to analyze the effect of a power-like shift, say $A(t)=c(1+t)^a$ with $c>0$, $a\in\R$, when the size of the domain is constant $L(t)=L_0$. In particular, it follows from  \eqref{Q+}, \eqref{Q-}, \eqref{overline-P}, \eqref{underline-P}, that
$$
\overline Q(t)=\frac{L_0}{2d}\max_{0\leq z\leq 1} \ddot{A}(t)z=\frac{L_0}{2d}\max(\ddot{A}(t),0), \quad  \underline Q(t)=\frac{L_0}{2d}\min_{0\leq z\leq 1} \ddot{A}(t)z=\frac{L_0}{2d}\min(\ddot{A}(t),0),$$
and $\overline P(t)=\underline P(t)=0$. 

First, we show that, in presence of a superlinear shift of the domain,  the population is doomed to extinction, regardless of how good the conditions are. 

\begin{corollary}[Superlinear shift]\label{cor:superlinear-shift}
Let the assumptions of Theorem \ref{th:sub-supersol} hold. Assume further that $L(t)=L_0$ for some $L_0>0$ and $A(t)=c(1+t)^a$ for some $c>0$ and $a>1$.

Then there are $C_1>0$, $C_2>0$ such that
$$
\Vert v(t,\cdot)\Vert _{L^\infty(0,L_0)}\leq C_1e^{-C_2 t^{2a-1}}, \quad \forall t>0,
$$
so that the solution uniformly goes to extinction at large times. 
\end{corollary}

\begin{proof} Since $\overline Q(t)=\ddot{A}(t)\frac{L_0}{2d}=\frac{ca(a-1)L_0}{2d}(1+t)^{a-2}$, we deduce from \eqref{sursol2} and some direct computations  that there is $C>0$ such that, for any $t>0$, any $0<y<L_0$,
\begin{equation}\label{par-dessus}
v(t,y)\leq C \exp\left(-\lambda t-\frac{c^2a^2}{4d(2a-1)}(1+t)^{2a-1}+\frac{L_0ca}{2d}(1+t)^{a-1}\right),
\end{equation}
from which the result follows since $2a-1>\max(1,a-1)$.
\end{proof}

Next, we show that a population that survives in a fixed domain (thanks to favorable enough conditions) would not be affected much by a sublinear shift of the domain (the other conditions being unchanged).

\begin{corollary}[Sublinear shift]\label{cor:sublinear-shift}  Let the assumptions of Theorem \ref{th:sub-supersol} hold. Assume further that $L(t)=L_0$ for some $L_0>0$ and $A(t)=c(1+t)^a$ for some $c>0$ and $a<1$, and that
\begin{equation}\label{cond-survie-fixed}
\lambda<0. 
\end{equation}

Then, for any $\ep>0$, there is $C>0$ such that 
$$
\min _{\ep\leq y\leq L_0-\ep} v(t,y)\geq C e^{-\lambda t}, \quad \forall t>0,
$$
so that the solution locally uniformly tends to infinity at large times.
\end{corollary}

\begin{proof} Observe that $\underline Q(t)=\ddot{A}(t)\frac{L_0}{2d}=\frac{ca(a-1)L_0}{2d}(1+t)^{a-2}$ if $0\leq a<1$ while $\underline Q(t)=0$ if $a<0$ so that, in any case, $\underline Q(t)\geq \frac{ca(a-1)L_0}{2d}(1+t)^{a-2}$. Hence, we deduce from \eqref{soussol2} and some direct computations that, for any $\ep>0$, there is $C>0$ such that, for any $t>0$, any $\ep<y<L_0-\ep$,
\begin{equation}\label{par-dessous}
v(t,y)\geq C  \exp\left(-\lambda t -\frac{c^2a^2}{4d(2a-1)}(1+t)^{2a-1}+\frac{L_0ca}{2d}(1+t)^{a-1}(1-\frac{y}{L_0})\right),
\end{equation}
from which the result follows since $1>\max(2a-1,a-1)$ (note that when $a=\frac 1 2$ the term $-\frac{c^2a^2}{4d(2a-1)}(1+t)^{2a-1}$ is obviously replaced by $-\frac{c^2a^2}{4d}\ln (1+t)$). 
\end{proof}

Hence, the condition \eqref{cond-survie-fixed} insures survival not only in a fixed domain ($a=0$) but still if the shift is sublinear ($a<1$). Note that \eqref{vp_estimate_final}  shows that \eqref{cond-survie-fixed} holds as soon as 
$$
r\geq \frac{d\pi^2}{L_0^2}+L_0^2\frac{\langle \alpha \rangle}{2}\left(\frac{2\pi^2-3}{6\pi^2}+\frac{\langle\alpha\beta^2\rangle}{\langle\alpha\rangle}-\frac{\langle\alpha\beta\rangle}{\langle\alpha\rangle} \right).
$$ 

Also, from the above proof, the critical case $\lambda=0$ insures survival (but not necessarily explosion) whenever $a<\frac 12$. As for the case $\frac 12 \leq a<1$, reproducing the arguments of Corollary \ref{cor:superlinear-shift}, one reaches a similar extinction result. In other words, the following holds. 

\begin{corollary}[Sublinear shift, critical case]\label{cor:sublinear-shift-critical}  Let the assumptions of Theorem \ref{th:sub-supersol} hold. Assume further that $L(t)=L_0$ for some $L_0>0$ and $A(t)=c(1+t)^a$ for some $c>0$ and $a<1$, and that
\begin{equation}\label{cond-survie-fixed-critical}
\lambda=0. 
\end{equation}
\begin{enumerate}
\item[(i)] Assume $a<\frac 12$. Then, for any $\ep>0$, there is $C>0$ such that 
$$
\min _{\ep\leq y\leq L_0-\ep} v(t,y)\geq C, \quad \forall t>0.
$$
\item[(ii)] Assume $\frac 12 \leq a<1$. Then there are $C_1>0$, $C_2>0$ such that, for all $t>0$,
$$
\Vert v(t,\cdot)\Vert _{L^\infty(0,L_0)}\leq \begin{cases}C_1e^{-C_2 t^{2a-1}}, &\text{ if } \frac 12 <a<1,\\
C_1e^{-C_2 \ln (1+t)}, &\text{ if } a=\frac 12,
\end{cases}
$$
so that the solution uniformly goes to extinction at large times. 
\end{enumerate}
\end{corollary}

The above asserts that  a population hardly surviving ($\lambda=0$) is very sensitive to shifts of the magnitude $(1+t)^{1/2}$.

Last, we consider the case of a linear shift $A(t)=c(1+t)$ ($c\geq 0$). We show that  a population that survives in a fixed domain (thanks to favorable enough conditions) would still survive when $c$ is small enough, but would go to extinction when $c$ is large enough.

\begin{corollary}[Linear shift]\label{cor:linear-shift} Let the assumptions of Theorem \ref{th:sub-supersol} hold. Assume further that $L(t)=L_0$ for some $L_0>0$ and $A(t)=c(1+t)$ for some $c\geq 0$, and that \eqref{cond-survie-fixed} holds. Define
$$
c^*:=2\sqrt{-\lambda d}>0.
$$

Then, if $0\leq c <c^*$, the solution locally tends to infinity at large times. Survival (but not necessarily explosion) still occurs if $c=c^*$. On the other hand, if $c>c^*$, the solution uniformly goes to extinction at large times.
\end{corollary}

\begin{proof} It suffices to use \eqref{par-dessus} and \eqref{par-dessous} in the case $a=1$. 
\end{proof}

\section{Numerical approach}\label{s:numerics}

In this section, we implement a numerical scheme to approximate the solution of the evolution problem \eqref{eq} in moving domains, with the goal of exploring various domain evolution types and their effects on population survival or extinction. Although transforming the problem onto a fixed reference domain is a possible approach, it leads to highly time and space dependent coefficients, which complicate both analysis and numerical implementation (particularly in higher dimensions). To avoid these difficulties, we work directly on the moving domain and employ the stabilized  space-time finite element method introduced in \cite{Moore_2018}. This approach simultaneously discretizes space and time,  reformulates  the problem as a diffusion-convection-reaction system in a non-cylindrical space-time domain, with the time derivative interpreted as a convection term in the extended space-time framework.


\subsection{Weak formulation}\label{ss:weak}

For some $T>0$, we consider the bounded and Lipschitz space-time domain
$$
Q:=(0,T)\times \Omega(t)\subset\R^2, \quad \Omega(t):=(A(t),A(t)+L(t)).
$$
The boundary of $Q$ is divided into three parts: the lateral boundary  $\Sigma=((0,T)\times \{A(t)\})\cup((0,T)\times\{A(t)+L(t)\})$, the bottom boundary  $\Sigma_0=\{0\}\times (A(0),A(0)+L(0))$ and the upper boundary  $\Sigma_T= \{T\}\times (A(T),A(T)+L(T))$. 

Let us define the Sobolev spaces
\begin{equation*}
H^{1,0}(Q):=\{u\in L^2(Q): \, u_{x}\in L^2(Q)\},
\end{equation*}
and 
\begin{equation*}
H^{1,1}(Q):=\{v\in L^2(Q): \, v_{x}\in L^2(Q), v_{t}\in L^2(Q)
\}.
\end{equation*}
For later purpose, let us consider,  on the domain $Q$,  the problem
\begin{equation}\label{eq_space_time}
\begin{cases}
u_t-du_{xx}-R(t,x)u=f,  &\quad \mbox{in} \,\,Q,\vspace{5pt}\\
u(t,x)=0, & \quad \mbox{on}\,\, \Sigma\cup \Sigma_0,
\end{cases}
\end{equation}
where 
$f$ is a given source function  in $L^2(Q)$ and 
\begin{equation} 
R(t,x)=r-\frac{\alpha(t)}{2}\left(x-A(t)-\beta(t)L(t)\right)^2.
\end{equation}
The space-time variational formulation of \eqref{eq_space_time} requires to find $u\in  H^{1,0}_{0,\underline{0}}(Q)$ such that 
\begin{equation}\label{weakPb}
-\int_Q u v_t \,dx dt +d\int_Q u_x v_x\,dx dt-\int_Q R(t,x)uv\,dxdt=\int_{Q} f v\,dx dt,\; \forall  v\in H^{1,1}_{0,\overline{0}}(Q),
\end{equation}
where the trial and test  spaces are  defined by
\begin{equation*}
    H^{1,0}_{0,\underline{0}}(Q):=\{u\in H^{1,0}(Q): u=0 \, \mbox{ on }\,\Sigma,  \, \mbox{ and }\, u=0  \mbox{ on } \,\Sigma_0  
\},
\end{equation*}
and 
\begin{equation*}
    H^{1,1}_{0,\overline{0}}(Q):=\{v\in H^{1,1}(Q): v=0 \, \mbox{ on }\,\Sigma,  \, \mbox{ and }\, v=0  \mbox{ on } \,\Sigma_T  
\}.
\end{equation*}

\begin{proposition}[Well-posedness]
Problem \eqref{weakPb} has a unique solution. 
\end{proposition}

\begin{proof}
We apply the same change of variables as in the previous section, see  \eqref{eq-v}, transforming  our problem into an equivalent one posed on the fixed spatial domain $
Q_0=(0,T)\times (0,L(0))$  whose boundary consists of the initial time boundary   $\widehat\Sigma_0=\{0\}\times(0,L(0))$, the lateral spatial boundary $\widehat\Sigma=((0,T)\times\{0\})\cup (0,T)\times\{L(0)\})) $ and the final time boundary $\widehat\Sigma_T=\{T\}\times(0,L(0))$. This is achieved by the  change of variable  $y=\frac{x-A(t)}{L(t)}L(0)$, where   $A$ and $L$ are in $C^2([0,+\infty))$ with  $L(t)>0$  for  all $t\in [0,T]$. Denoting the unknown function in the new coordinates by $u=u(t,y)$, the problem becomes that of finding $u$ satisfying:
$$
u\in H^{1,0}_{0,\underline{0}}(Q_0):=\{u\in L^2(Q_0): \, u_{y}\in L^2(Q_0),\, u=0\, \mbox{ on } \,\widehat\Sigma, \, \mbox{ and }  u=0\, \mbox{ on } \,\widehat\Sigma_0
\}
$$
such that 
\begin{equation}\label{Pb_weak}
a(u,v)=l(v),  \quad \forall v\in H^{1,1}_{0,\overline{0}}(Q_0),
\end{equation} 
where $H^{1,1}_{0,\overline{0}}(Q_0):=\{v\in L^2(Q_0): \, v_{y}, v_t\in L^2(Q_0),\,  v=0\, \mbox{ on } \,\widehat\Sigma, \, \mbox{ and }  v=0\, \mbox{ on } \,\widehat\Sigma_T
\}$,
\begin{multline*}
a(u,v):=\int_{Q_0} -u v_t\,dy dt +\int_{Q_0}\frac{dL^2(0)}{L^2(t)} u_y v_y\,dy dt\\-\int_{Q_0} \frac{\dot{A}(t)L(0)+y\dot{L}(t)}{L(t)}u_y v\,dy dt-\int_{Q_0} \tilde{R}(t,y)uv\,dy dt,
\end{multline*} 
with
\begin{equation*}
\tilde{R}(t,y)=r-\frac{\alpha(t) L^2(t)}{2L^2(0)}\left(y-\beta(t) L(0)\right)^2,
\end{equation*}
and 
\begin{equation*}
l(v)=\int_{Q_0} f v\,dy dt.
\end{equation*}
 The existence and uniqueness for this problem is well established, see \cite[Chapter III, Theorems 3.1 and 3.2]{lohwater2013boundary}.
\end{proof}

\subsection{Space-time finite element discretization} \label{ss:discretization}

 In this subsection, we aim to construct a \emph{continuous Galerkin finite element scheme} to numerically approximate the solution of the problem
\begin{equation}\label{NonHomog_eq_space_time}
\begin{cases}
u_t - d u_{xx} - R(t,x) u = 0, & \quad \text{in } Q, \vspace{5pt}\\
u = 0, & \quad \text{on } \Sigma, \\
u = u_0, & \quad \text{on } \Sigma_0,
\end{cases}
\end{equation}
where the non-homogeneous Dirichlet condition $u_0$ is assumed to satisfy \( u_0 \in H^1(\Sigma_0) \), and its tangential derivative (namely $\partial _x u_0$) belongs to \( H^{1/2}(\Sigma_0) \). In the sequel we denote $u_0^*$ the function defined on $\partial Q$ by $u_0^*:=\mathbf{1}_{\Sigma_0}u_0$, so that $u_0^*\in H^1(\partial Q)$. 

We would like to perform a lifting $g=g(t,x)$ of the boundary condition $u_0^*$. It is known that the regularity of $g$ depends on the regularity of the domain. In \cite{Grisvard1985}, Grisvard showed that if the boundary \( \partial Q \) is of class \( \mathcal{C}^{m-1,1} \), then the lifting belongs to \( H^k(Q) \) for all \( k \leq m \). In particular, for \( m = 2 \), we have \( g \in H^2(Q) \). However, in our setting, the moving spatial domain is only Lipschitz and, in this case, the existence of a $H^2$ lifting requires a compatibility condition, see \cite[Theorem 3]{Geymonat2000} by Geymonat and Krasucki, that, in our setting, reads  as $\partial _x u_0 \in H^{1/2}(\Sigma _0)$, which we have precisely assumed.

Hence, from \cite[Theorem 3]{Geymonat2000}, we are equipped with a lifting function \( g \in H^2(Q) \) such that \( \gamma_0(g) = u_0^* \). Setting \( u = w + g \), we transform problem \eqref{NonHomog_eq_space_time} into the following equivalent problem for \( w \):
\begin{equation}\label{eq_space_time_homogeneBis}
\begin{cases}
w_t - d w_{xx} - R(t,x) w = f, & \quad \text{in } Q, \vspace{5pt}\\
w = 0, & \quad \text{on } \Sigma, \\
w = 0, & \quad \text{on } \Sigma_0,
\end{cases}
\end{equation}
where the right-hand side is given by  
\[
f = -\partial_t g + d\, \partial_{xx} g + R g \in L^2(Q).
\]

\begin{remark}
Problems of this nature are often approached by decoupling time and space in the discretization process. A typical example is the use of an Euler scheme in time combined with finite elements in space. However, when dealing with time-dependent domains, this approach becomes unsuitable.
To overcome this limitation, we discretize both time and space simultaneously by considering the problem within the space-time domain $Q$. In this setting, the time derivative is treated as a convection term in the time direction, resulting in a diffusion-convection-reaction equation. Such problems are well-known for their numerical instability when using conventional schemes.
Following the stabilization approach introduced  by Hughes {\it et al.} in \cite{Hughes1982} for convection-diffusion-reaction problems, where upwind test functions are used in the context of the Stream Upwind Petrov-Galekin (SUPG) method, Moore adapted these ideas in \cite{Moore_2018}  to the context of parabolic homogeneous initial-boundary value problem  on moving domains. Specifically, Moore proposes using test functions of the form \(v_h + \theta h \partial_t v_h \), where \( \theta > 0 \) is a stabilization parameter to be specified. 
\end{remark}

In the following,  we parallel  the method described in \cite{Moore_2018} to our specific problem \eqref{eq}. 
The first step consists in defining a triangulation  $\mathcal{K}_h$ of the space-time domain $Q$, into non-degenerate triangles.  For each $K\in \mathcal{K}_h$, let  $h_K$ denote  the diameter of the element,  and define the global mesh size $h$  as 
$$
h:=\max \{ h_K:  K\in\mathcal{K}_h \} .  
$$
We further assume that the  triangulation is quasi-uniform,  i.e. there is $C>0$ such that 
\begin{equation}\label{quasi-uniform estimate}
h_K\leq h\leq C h_K, \quad \forall K\in \mathcal{K}_h.
\end{equation}  
Let $K_i, K_j\in \mathcal{K}_h$ be two neighboring triangles, and   consider the interior facet 
\begin{equation*}
F_{ij}:=\overline{K_i}\cap \overline{K_j}.
\end{equation*}
We   define $\mathcal{F}_I$ as the set of all interior facets of $\mathcal{K}_h$,  namely
\begin{equation*}
\mathcal{F}_I:=\left(\bigcup_{i \in I} \partial K_i \right) \backslash \partial Q.
\end{equation*}
At this point, we introduce the discrete space-time space $V_{0h}$. To do  this,  let  $\mathbb{P}_2$ denote  the set of  polynomials of degree less than  or equal to $2$, and define the space  $V_h$ as   
\begin{equation*}
V_h:=\left\{v_h\in C^0(\overline{Q}): v_h\vert_K\in\mathbb{P}_2(K),\,  \forall K\in\mathcal{K}_h \right\}.
\end{equation*}    
We then define the discrete space-time space $V_{0h}$ as    
\begin{equation*}
V_{0h}:=V_h \cap H^{1,1}_{0,\underline{0}}(Q).
\end{equation*}
Before deriving the  finite element scheme,  we recall below some notation and jump properties introduced in \cite{Moore_2018}, which will be used throughout the formulation. 

\medskip

\noindent {\bf Notations.} Denote ${\bf n}_i=(n_{i,x},n_{i,t})^\intercal $ the outer unit normal vector with respect to $K_i$.  For  a sufficiently smooth scalar function $v$, we will denote by $v_i$, $v_j$ the traces of the function $v$ on $F_{ij}\in \mathcal {F}_I$  an interior edge, and the jump across the interior edge of $\mathcal{F}_I$ is defined  by
\begin{equation*}
 \llbracket v \rrbracket:= v_i {\bf n}_i+v_j{\bf n}_j.
\end{equation*}
The jump in space direction is given by
\begin{equation*}
\llbracket v \rrbracket_x:= v_i {n}_{i,x}+v_j{n}_{j,x},
\end{equation*}
whereas the jump in time direction is  defined  by
\begin{equation*}
\llbracket v \rrbracket_t:= v_i {n}_{i,t}+v_j{ n}_{j,t}.
\end{equation*}
The average of a function on the interior edge  $F_{ij}$ is
\begin{equation*}
\{v\}:=\frac{1}{2}\left(v_i+v_j\right), 
\end{equation*}
and the upwind  value in  time direction is given by: 
\begin{equation*}
\{v\}^{up}:=\left\{
\begin{array}{ll}
v_i & \mbox{  for }  n_{i,t}\geq 0, \\
v_j & \mbox{ for }  n_{i,t}<0.
\end{array}
\right.
\end{equation*} 
Finally,  the downwind value in the time direction is given by  
\begin{equation*}
\{v\}^{down}:=\left\{
\begin{array}{ll}
v_j & \mbox{ for } n_{i,t}\geq 0, \\
v_i & \mbox{ for } n_{i,t}<0.
\end{array}
\right.
\end{equation*}  

\medskip

\noindent {\bf Jump properties.} The following properties of the jump of a product of functions  will play a crucial role in the subsequent analysis of the scheme:  if  $F_{ij}\in \mathcal{F}_I$ is an interior edge, and if  $u$ and  $v$ are sufficiently smooth  functions on the interface, there hold
\begin{equation}\label{jump_x}
\llbracket uv \rrbracket_x \,=\, \{u\}\llbracket v \rrbracket_x + \{v\}\llbracket u\rrbracket_x, 
\end{equation} 
\begin{equation}\label{jump_t}
\llbracket uv \rrbracket_t \,=\, \{u\}^{up}\llbracket v \rrbracket_t + \{v\}^{down}\llbracket u\rrbracket_t, 
\end{equation}
and 
\begin{equation}\label{v_up-v^2}
\{v\}^{up}\llbracket v \rrbracket_t -\frac 12 \llbracket v^2 \rrbracket_t \, = \,\frac 12 \lvert n_{i,t}\rvert\llbracket v \rrbracket^2.
\end{equation}

\medskip

 Let us emphasize again that the following is deeply inspired by \cite{Moore_2018} to which some computations are borrowed for completeness.
To  approximate the solution of \eqref{eq_space_time_homogeneBis}, let us consider the problem of finding  $w\in H^{1,1}_{0,\underline{0}}(Q)$ such that 
\begin{equation}\label{eq_faible_lift}
\int_Q \partial_t w \,v \,dxdt+\int_Q d \partial_x w\,\partial_x v \, dxdt- \int_Q R(t,x)w\,v\,dxdt=l(v),\quad \forall v\in H^{1,1}_{0,\underline{0}}(Q), 
\end{equation}
with 
\begin{equation*}
l(v):=-\int_Q \partial_t g\,v\,dxdt-\int_Q d\partial_x g\, \partial_x v\,dxdt+\int_Q R(t,x)g\,v\,dxdt.
\end{equation*}
Next, we use  test functions 
 of the form $v_h+\theta h\, \partial_t v_h$, where  $v_h\in V_{0h}$ is arbitrary   and   $\theta>0$ is a positive constant. Then the space-time variational  formulation  (\ref{eq_faible_lift})  reads as
 \begin{multline*}
\int_Q \partial_t w\left(v_h+\theta h \partial_t v_h\right)\,dxdt+\int_Q \Big( d\partial_{x}w\partial_x v_h-d\theta h\,\partial_{xx}w \partial_t v_h\\- R(t,x)w\left(v_h+\theta h \partial_t v_h\right)\Big)\,dxdt=l\left(v_h+\theta h \partial_t v_h\right).
\end{multline*}
Next, summing  on each element of $\mathcal{K}_h$ and integrating  by parts with respect to spatial direction we get
\begin{multline*}
-\int_Q \partial_{xx}w\,\, \partial_t v_h\, dxdt\\=\sum_{K\in \mathcal{K}_h} \int_K \partial_x w \,\, \partial_{xt}v_h\, dxdt-\sum_{F_{ij}\in \mathcal{F}_I} \int_{F_{ij}} \llbracket \partial_x w\,\,\partial_t v_h \rrbracket_x \,ds-\int_{\Sigma} n_x\left(\partial_x w\,\,\partial_t v_h\right)\,ds.
\end{multline*} 
By performing another integration by parts with respect to time and  since $\partial_x v_h =0$ on $\Sigma_0$, we obtain
\begin{multline*}
-\int_Q \partial_{xx}w\,\, \partial_t v_h\, dxdt=-\sum_{K\in \mathcal{K}_h} \int_K \partial_{tx} w \,\, \partial_{x}v_h\, dxdt+\sum_{F_{ij}\in \mathcal{F}_I} \int_{F_{ij}}\llbracket \partial_x w\,\,\partial_x v_h \rrbracket_t \, ds \\
+\int_{\Sigma\cup \Sigma_T} n_t\left(\partial_x w \,\,\partial_xv_h\right)\,ds -\sum_{F_{ij}\in \mathcal{F}_I} \int_{F_{ij}} \llbracket \partial_x w\,\,\partial_t v_h \rrbracket_x \,ds-\int_{\Sigma} n_x\left(\partial_xw\,\,\partial_t v_h\right)\,ds.
\end{multline*}
Considering the terms on the interior  facets $F_{ij} \in \mathcal{F}_I$, we use  properties on the jump of a product of functions, namely \eqref{jump_x} and \eqref{jump_t}, to obtain
\begin{align*}
\sum_{F_{ij}\in \mathcal{F}_I} \int_{F_{ij}}\llbracket \partial_x w\,\,\partial_x v_h \rrbracket_t - \llbracket \partial_x w\,\,\partial_t v_h \rrbracket_x \,ds=&\sum_{F_{ij}\in \mathcal{F}_I} \int_{F_{ij}} \Big(\{\partial_x w\}^{up}\llbracket \partial_x v_h \rrbracket_t+\{\partial_x v_h\}^{down}\llbracket \partial_x w\rrbracket_t\\
&-\{\partial_x w\}\llbracket \partial_t v_h\rrbracket_x -\{\partial_t v_h\}\llbracket \partial_x w\rrbracket_x \Big)\, ds.
\end{align*}  
 Assuming that the solution $w$ belongs to $H^2(Q)$ allows  us to further simplify those boundary terms since  the jumps  $\llbracket \partial_x w\rrbracket_x$ and $\llbracket \partial_x w\rrbracket_t$ are null.  Thus we have 
\begin{equation*}
\sum_{F_{ij}\in \mathcal{F}_I} \int_{F_{ij}}\left(\llbracket \partial_x w\,\,\partial_x v_h \rrbracket_t - \llbracket \partial_x w\,\,\partial_t v_h \rrbracket_x \right)\,ds=\sum_{F_{ij}\in \mathcal{F}_I} \int_{F_{ij}} \left(\{\partial_x w\}^{up}\llbracket \partial_x v_h \rrbracket_t-\{\partial_x w\}\llbracket \partial_t v_h\rrbracket_x\right) \,ds.
\end{equation*} 
Next, we show that $n_t \,\partial_x v_h-n_x\,\partial_t v_h=0$. First,  notice that we can rewrite the left side as $\nabla v_h \cdot\left(-n_x,n_t\right)^\intercal$. Now,  consider that $\left(-n_x,n_t\right)^\intercal$ is the rotation by an angle $\frac{\pi}{2}$ of the outer unit normal vector, so $\left(-n_x,n_t\right)^\intercal$ is a tangential vector on $\Sigma$. Since $v_h=0$ on $\Sigma$,   we have  $n_t \,\partial_x v_h-n_x\,\partial_t v_h=0$, resulting in the suppression of the boundary term          
\begin{equation*}
\int_\Sigma \partial_x w\left(n_t \,\partial_x v_h-n_x\,\partial_t v_h\right)\,ds=0. 
\end{equation*}  
Finally, assuming that $w\in H^2(Q)$ allows us to write $\llbracket \partial_tw\rrbracket =0$. Consequently,  it is harmless to add, for $\delta>0$, the consistent term
\begin{equation*}
\theta h \sum_{F_{ij}\in\mathcal{F}_I} \int_{F_{ij}}d\{\partial_x v_h\}\llbracket \partial_t w\rrbracket_x\, ds+ \delta \sum_{F_{ij}\in\mathcal{F}_I} \int_{F_{ij}}\llbracket \partial_t w\rrbracket_x\llbracket \partial_t v_h\rrbracket_x\,ds.
\end{equation*}

 \begin{remark} 
 The stabilization term $\theta h \, \partial_t w \, \partial_t v$
corrects instabilities caused by advection-dominance in the time direction, with $\theta$ controlling artificial diffusion and $h$ scaling it with the mesh size. A penalty term
$\delta \llbracket \partial_t w \rrbracket \llbracket \partial_t v \rrbracket$ enforces continuity of the time derivative across element interfaces, enhancing accuracy and convergence. These techniques follow the SUPG framework, as developed in \cite{Hughes1987} and \cite{Johnson1987}.
\end{remark}

Putting everything together, we are now in the position to  write the variational space-time finite element scheme: it consists in finding $w_h\in V_{0h}$ such that 
\begin{equation}\label{pb_discret}
a_h(w_h,v_h)=l(v_h),  \quad \forall v_h \in V_{0h},
\end{equation}
where
\begin{eqnarray*}     
a_h(w_h,v_h)&:=&\int_Q \partial_t w_h \left(v_h+\theta h\,\partial_t v_h\right)\,dxdt+\int_Q d\, \partial_x w_h\,\, \partial_x v_h \, dxdt \\&&-\int_Q R(t,x)\,w_h\left(v_h+\theta h \,\partial_t v_h\right)\, dxdt
-d\theta h\sum_{K\in \mathcal{K}_h} \int_K \partial_{tx} w_h\,\,\partial_{x}v_h\, dxdt\\&&+d\theta h\sum_{F_{ij}\in \mathcal{F}_I} \int_{F_{ij}} \{\partial_x w_h\}^{up}\llbracket \partial_x v_h \rrbracket_t-\{\partial_x w_h\}\llbracket \partial_t v_h\rrbracket_x\,ds
\\&&+d\theta h\int_{\Sigma_T} \partial_x w_h\,\,\partial_x v_h \,ds+\theta h \sum_{F_{ij}\in\mathcal{F}_I} \int_{F_{ij}}d\{\partial_x v_h\}\llbracket \partial_t w_h\rrbracket_x\, ds\\&&+\delta\,\sum_{F_{ij}\in\mathcal{F}_I} \int_{F_{ij}}\llbracket \partial_t w_h\rrbracket_x\llbracket \partial_t v_h\rrbracket_x\,ds,
\end{eqnarray*} 
and 
\begin{multline*}
l_h(v_h):=-\int_Q \partial_t g\left(v_h+\theta h\,\partial_t v_h\right)\,dxdt\\-\int_Q d\partial_x g\, \partial_x  \left(v_h+\theta h\,\partial_t v_h\right)\,dxdt+\int_Q R(t,x)g \left(v_h+\theta h\,\partial_t v_h\right)\,dxdt.
\end{multline*}
Now that our strategy is outlined, our next objective is to establish the well-posedness of the discrete problem.  

\begin{theorem}[Well-posedness]\label{th:discret well posedness}
Assume that $\theta$ is small enough. Then the  discrete problem \eqref{pb_discret} has a unique solution in $V_{0h}$.  
\end{theorem}

We first need the following result.

\begin{lemma}
The space  \( V_{0h} \) equipped with
\begin{equation}\label{MeshDependentNorm}
    \begin{aligned}
\lvert\lvert v_h\rvert\rvert_h\,=\Big(\,&\lvert\lvert \partial_x v_h\rvert\rvert^2_{L^2(Q)}+\theta h \lvert\lvert \partial_t v_h\rvert\rvert^2_{L^2(Q)}+\lvert\lvert  v_h\rvert\rvert^2_{L^2(\Sigma_T)}+\theta h\lvert\lvert \partial_x v_h\rvert\rvert^2_{L^2(\Sigma_T)}\\
&+\theta h\sum_{F_{i,j}\in\mathcal{F}_I}\lvert\lvert \llbracket \partial_x v_h\rrbracket_t\rvert\rvert^2_{L^2(F_{i,j})}+\delta \sum_{F_{i,j}\in\mathcal{F}_I}\lvert\lvert \llbracket \partial_t v_h\rrbracket_x\rvert\rvert^2_{L^2(F_{i,j})}\Big)^{1/2},
\end{aligned}
\end{equation}
is a Hilbert space.
\end{lemma}

\begin{proof} 
 The application $\|\cdot\|_h$  is a norm on $V_{0h}$ (arising from an obvious scalar product). Indeed,  $\lvert\lvert v_h \rvert\rvert_h=0$   implies that  $\partial_x v_h=0$, $\partial_t v_h=0$,  $  v_h=0$ on $\Sigma_T$.  If all these terms are zero,  and since $v_h=0 $ on $\Sigma\cup\Sigma_0$, then $v_h=0$  throughout the domain $Q$. Therefore $(V_{0h},\|\cdot\|_h)$ is a Hilbert space, as it is  finite-dimensional. 
\end{proof}

\begin{remark} The mesh-dependent norm \eqref{MeshDependentNorm} explicitly includes the stabilization   to counteract numerical oscillations in the advection direction. Standard norms (such as those of $L^2$ or $H^1$) may not  handle optimal convergence rates due to these oscillations (see, for example, \cite{ErnGuermond2004} and \cite{Hughes1986}). Using this   $h$-dependent norm is crucial  for accurately approximating solutions in  advection-dominated problems, as it captures the stabilizing effects of the discretization, provides more consistent error estimates  and ensures the numerical solution remains stable under mesh refinement (see \cite{AinsworthOden2000}).    
 \end{remark}
 

Let us now  observe that if $u$ solves \eqref{NonHomog_eq_space_time},  then  $U(t,x):=u(t,x)e^{-ct}$ obviously solves
\begin{equation*}
\begin{cases}
U_t-d U_{xx}+(c-R(t,x))U= 0, &\quad \mbox{in} \,\,Q,\vspace{5pt}\\
U=0, & \quad \mbox{on}\,\, \Sigma,\\
U=u_0, & \quad \mbox{on} \,\, \Sigma_0. 
\end{cases}
\end{equation*}
Hence, recalling that  $R\leq r$ on $Q$, we deduce that,  up to a change of unknown function if necessary, we can  assume the existence of a constant  $\rho>0$ such that 
\begin{equation}\label{parabolic_trick}
0<\rho\leq - R(t,x),\; \forall (t,x)\in Q.
\end{equation}  
We are now in position to prove the $V_{0h}$-ellipticity of the bilinear form $a_h$.

\begin{lemma}  If $\theta>0$ is small enough, the bilinear form $a_h$ in   \eqref{pb_discret}  is $V_{0h}$-elliptic.
\end{lemma}

\begin{proof}
For   $v_h\in V_{0h}$, we have
\begin{align*}
a_h(v_h,v_h)=&\int_Q v_h\,\partial_t v_h\,dxdt+\theta h \norm{\partial_t v_h}{L^2(Q)}^2+d\norm{\partial_x v_h}{L^2(Q)}^2\\
&-\int_Q R(t,x)\,v_h\left(v_h+\theta h \,\partial_t v_h\right)\, dxdt-d\theta h\sum_{K\in \mathcal{K}_h} \int_K \partial_{tx} v_h\,\,\partial_{x}v_h\, dxdt\\
&+d\theta h\sum_{F_{ij}\in \mathcal{F}_I} \int_{F_{ij}} \{\partial_x v_h\}^{up}\llbracket \partial_x v_h \rrbracket_t\,ds+d\theta h\norm{\partial_x v_h}{L^2(\Sigma_T)}^2\\& +\delta\sum_{F_{ij}\in \mathcal{F}_I} \int_{F_{ij}}{\llbracket \partial_t v_h\rrbracket_x}^2\, ds.
\end{align*}
First,  using the estimate \eqref{parabolic_trick} and  applying the divergence theorem, we get:
\begin{align*}
a_h(v_h,v_h)\,\geq\, &\frac12\int_{\partial Q} v_h^2\,n_t\,dxdt+\theta h \norm{\partial_t v_h}{L^2(Q)}^2+d\norm{\partial_x v_h}{L^2(Q)}^2+\rho\norm{v_h}{L^2(Q)}^2\\
&+\frac{\rho\theta h}{2} \int_{\partial Q} v_h^2\,n_t\,dxdt-\frac{d\theta h}{2}\sum_{K\in \mathcal{K}_h} \int_{\partial K} \partial_{x} v_h^2n_{i,t}\, dtdx\\
&+d\theta h\sum_{F_{ij}\in \mathcal{F}_I} \int_{F_{ij}} \{\partial_x v_h\}^{up}\llbracket \partial_x v_h \rrbracket_t\,ds+d\theta h\norm{\partial_x v_h}{L^2(\Sigma_T)}^2 \\& +\delta\sum_{F_{ij}\in \mathcal{F}_I} \norm{\llbracket \partial_t v_h\rrbracket_x}{L^2(F_{ij})}^2.
\end{align*}
Rewriting the boundary terms  such that  $\partial Q =\Sigma\cup\Sigma_0\cup\Sigma_T$ and  $\mathcal{F}_I=(\cup_{K\in \mathcal{K}_h}\partial K)\backslash\partial Q$ with interior facet $F_{ij}\subset \mathcal{F}_I$ and using $v_h=0$ on $\Sigma\cup \Sigma_0$ yields
\begin{align*}
a_h(v_h,v_h)\,\geq\, &\frac12\norm{v_h}{L^2(\Sigma_T)}^2+\theta h \norm{\partial_t v_h}{L^2(Q)}^2+d\norm{\partial_x v_h}{L^2(Q)}^2+\rho\norm{v_h}{L^2(Q)}^2+\frac{\rho\theta h}{2}  \norm{v_h}{L^2(\Sigma_T)}^2\\
&-\frac{d\theta h}{2}\int_{\Sigma}n_{t}\partial_{x} v_h^2\,ds+d\theta h\sum_{F_{ij}\in \mathcal{F}_I} \int_{F_{ij}} \{\partial_x v_h\}^{up}\llbracket \partial_x v_h \rrbracket_t-\frac12 \llbracket \partial_{x} v_h^2\rrbracket_t\,ds\\
&+\frac{d\theta h}{2}\norm{\partial_x v_h}{L^2(\Sigma_T)}^2 +\delta\sum_{F_{ij}\in \mathcal{F}_I} \norm{\llbracket \partial_t v_h\rrbracket_x}{L^2(F_{ij})}^2.
\end{align*} 
Next, regrouping every term and applying  the equality \eqref{v_up-v^2} we reach
\begin{align*}
a_h(v_h,v_h)\,\geq\, &\frac{1+\rho\theta h}{2}\norm{v_h}{L^2(\Sigma_T)}^2+\theta h \norm{\partial_t v_h}{L^2(Q)}^2+d\norm{\partial_x v_h}{L^2(Q)}^2+\rho\norm{v_h}{L^2(Q)}^2\\
&-\frac{d\theta h}{2}\norm{\partial_{x} v_h}{L^2(\Sigma)}^2+\frac{d\theta h}{2}\sum_{F_{ij}\in \mathcal{F}_I} \int_{F_{ij}}\lvert n_{i,t}\rvert \llbracket \partial_x v_h \rrbracket^2 ds\\
&+\frac{d\theta h}{2}\norm{\partial_x v_h}{L^2(\Sigma_T)}^2 +\delta\sum_{F_{ij}\in \mathcal{F}_I} \norm{\llbracket \partial_t v_h\rrbracket_x}{L^2(F_{ij})}^2.
\end{align*}
Since $\lvert n_{i,t}\rvert\geq n_{i,t}^2 $, we get:
\begin{equation*}
a_h(v_h,v_h)\geq \min(\frac{1}{2},d) \norm{v_h}{h}^2-\frac{d\theta h}{2}\norm{\partial_{x} v_h}{L^2(\Sigma)}^2.
\end{equation*} 
 Finally, the inverse inequality (see, for example, \cite[Theorem 4.2]{Evans2013}) provides a constant $\widetilde C>0$ such that
\begin{equation*}
\lvert\lvert  \partial _x  v_h\rvert\rvert_{L^2(\partial K)}\leq \widetilde C h_K^{-\frac 12}\lvert\lvert\partial_xv_h\rvert\rvert_{L^2(K)}, \mbox{ for } K\in \mathcal{K}_h,
\end{equation*}
and using the quasi-uniform mesh assumption \eqref{quasi-uniform estimate}, 
we obtain
\begin{align*}
a_h(v_h,v_h)&\geq \min(\frac{1}{2},d)\norm{v_h}{h}^2-\frac{d\theta \widetilde C^2 C}{2}\norm{\partial_x v_h}{L^2(Q)}^2\\
&\geq \left(\min(\frac{1}{2},d)-\frac{d\theta \widetilde C^2 C}{2}\right)\norm{v_h}{h}^2,
\end{align*}
which, for $\theta>0$ small enough, proves  the $V_{0h}$-coercivity of $a_h$ with respect to the norm $\|\cdot\|_h$.
\end{proof}

\begin{proof}[Proof of  Theorem \ref{th:discret well posedness}]  Apply the Lax-Milgram lemma.
\end{proof}

\subsection{Numerical simulations} \label{ss:numerics}
In this subsection, 
we implement the discretization of problem \eqref{eq} in FreeFem++ (\url{http://www.freefem.org}), a high level, free software package. 
To validate our implementation, we first consider a simplified test case on a fixed domain where $A(t)=0$, $L(t)=L$, and no   selection occurs ($\alpha\equiv 0$). In this configuration, the solution to \eqref{eq} with initial condition  $u_0(x)=\sin \left(\frac{\pi x}{L}\right)$ is  explicitly given by 
\begin{equation*}
u(t,x)=\sin\left(\frac{\pi x}{L}\right)e^{(r-\frac{d\pi^2}{L^2})t}.
\end{equation*}
This exact solution serves as a benchmark to verify the accuracy of our numerical scheme, resumed in Figure \ref{fig: error}.

\begin{figure}[h!]
\centering
\includegraphics[scale=0.3]{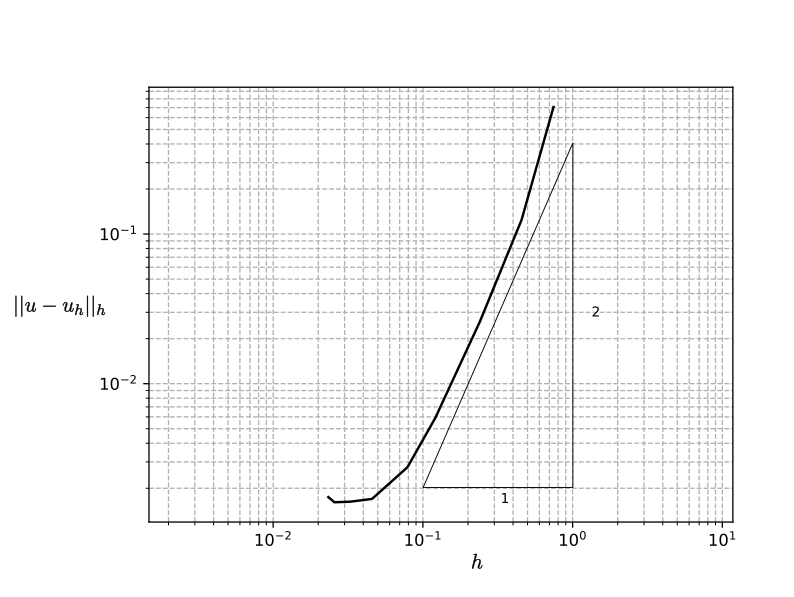}      
\caption{\textbf{Evolution of the consistency error ($\norm{u-u_h}{h}$) with respect to the mesh size $h$.} The parameters are as set as the following, $r=6$, $L=1.5$, $\theta=10^{-5}$ and $\delta=10^{-6}$.
  }
\label{fig: error}
  \end{figure} 
We now present several  numerical tests  with parameters 
$\theta =10^{-5}$, $\delta =10^{-6}$,  and initial data $u_0(x)=\sin(\frac{\pi x}{L})$, to highlight some underlying properties of the model.\par 
To begin,  we simulate the evolution of the solution to \eqref{eq-fix} to gain new insights into the survival dynamics for a fixed domain with a seasonal optimum position.

\begin{example}[Seasonal optimum's position]\label{ex:seasonal} Here we consider the test case $d=10$, $r=17.5$, $\alpha(t)=5$, $\beta(t)=\frac{1}{2}+\frac{1}{2}\sin(2t)$, $L=3$, which, as easily checked, meets the  sufficient conditions of survival as stated in Corollary \ref{cor:sur}. The outcomes are presented in Figure \ref{fig:seasonal}. 
\begin{figure}[h!]
\centering
\includegraphics[scale=0.4]{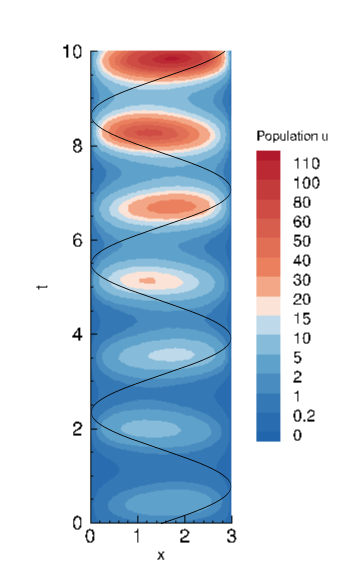}      
\caption{\textbf{Evolution of the solution with a time-periodic optimum in a seasonal model.} The parameters are as described in Example \ref{ex:seasonal}. In solid black lines, the evolution of the position of the optimum $x_{opt}(t)=\beta (t)L$.
  }
\label{fig:seasonal}
  \end{figure}
  As expected, survival takes place. Furthermore, the simulation showcases subtle interaction between the time periodic position of the optimum and the behavior (local growth/decay) of the population. Indeed, while the optimum position periodically describes the whole interval $(0,L)$, the position of the maximum of the solution oscillates on shorter intervals avoiding the boundaries (recall the zero Dirichlet boundary conditions) and decreasing as time passes. In other words, we observe damped oscillations of the position of the maximum. Interestingly, we also observe that $t\mapsto \Vert u(t,\cdot)\Vert_{L^\infty}$ is non monotonic. We refer to the pattern in Figure \ref{fig:seasonal}. 
\end{example}
\begin{example}[Sub-linear shift]\label{ex:sublinear} 
We  consider   a domain  undergoing a sub-linear shift while maintaining a constant length, namely $A(t)=\sqrt{t}$ and $L(t)=2$. The remaining  parameters are set as follows:  $T=10$, $d=1$, $r=4.1$, $\alpha(t)=7+0.2\sin(t)$, and either $\beta_l(t)=0.2-0.1\sin(4t)$ (left panel) or $\beta_r(t)=0.8+0.1\sin(4t)$ (right panel).  The results are presented  in Figure \ref{fig:sqrt}. This simulation illustrates that survival is not affected by a constant-length domain  that shifts sub-linearly, see Corollary \ref{cor:sublinear-shift}. Let us also note that the position of the maximum of $x\mapsto u(t,x)$ does not coincide with the optimal  position $x_{opt} (t)$. Furthermore, the habitat shifting towards right, we observe that, even if survival occurs in both cases, the optimum leaning on the left, via $\beta_l(t)$, is more favorable to the population than the optimum leaning on the right, via $\beta_r(t)$. 
\begin{figure}[h!]
\centering
\includegraphics[scale=0.6, height=5cm]{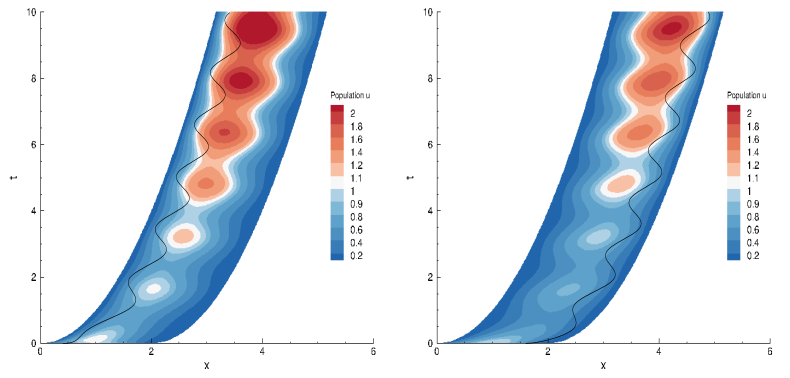}       
\caption{
\textbf{Effect of a sub-linear domain shift and of the position of the optimum.} The parameters are as described in Example \ref{ex:sublinear}. In solid black lines, the evolution of the position of the optimum $x_{opt}(t)=A(t)+\beta (t)L$.
}
\label{fig:sqrt}
\end{figure}
\end{example}

\begin{example}[Super-linear shift]\label{ex:superlinear} 
 We  consider a domain that shifts super-linearly   while maintaining  a constant length,  namely  $A(t)=t^{1.35}$ and $L(t)=35$. The parameters are as follows:  $T=30$, $d=7$, $r=1.25$, $\alpha (t)=0.4+0.04\sin(t)$, $\beta(t)=0.5+0.1\sin(t)$.  The results  are presented in Figure \ref{fig:superlinear}. We observe extinction of the population as explained by Corollary \ref{cor:superlinear-shift}: the domain translates \lq\lq too rapidly'' for the population. However, the simulation provides insight into the transient dynamics: the population  takes advantage of the periodic return of the optimum in its neighborhood to perform a small rebound (which  is not enough to survive at large times). This  contrasts with the above sub-linear case, where the domain shift merely relocates the solution maximum without causing extinction. 
\begin{figure}[h!]
   \centering
\includegraphics[scale=3, height=7cm]{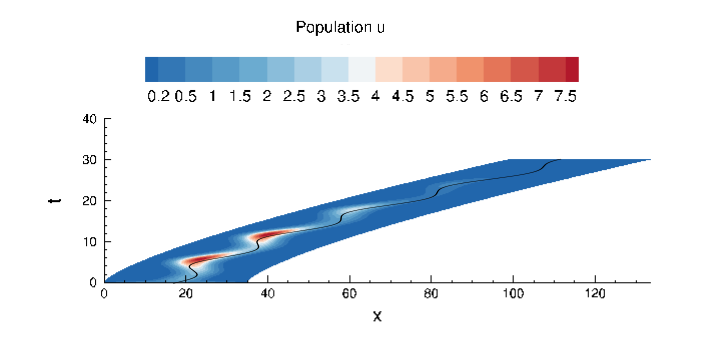 }  
\caption{\textbf{Population extinction due to a super-linear domain shift.} The parameters are as described in Example \ref{ex:superlinear}.  In solid black lines, the evolution of the position of the optimum $x_{opt}(t)=A(t)+\beta (t)L$.
}
\label{fig:superlinear}
\end{figure}
\end{example}
\begin{example}[Linear shift]\label{ex:linear} 
We consider  a domain that  shifts linearly  while  maintaining a constant length, namely $A(t)=2t$ and $L(t)=3$. The parameters are set as follows:  $d=1.6$, $T=4$, $r=4.8$, $\alpha (t)=5+0.1\sin(4t)$.  The results  are presented in  Figure \ref{fig:lin}. Let us recall that, from Corollary \ref{cor:linear-shift}, the comparison between $c=2$ and $c^*=2\sqrt{-\lambda d}$ decides between survival or extinction, where $\lambda$ obviously depends on the optimum position and thus on $\beta(t)$. To examine  this,  we consider three different cases for $\beta(t)$, namely   $\beta_l(t)=0.1-0.1\sin(t)$ (left panel) for which $\langle \beta_l \rangle=0.1$, $\beta_r(t)=0.9+0.1\sin(t)$ (right panel) for which $\langle \beta_r \rangle=0.9$, and $\beta_m(t)=0.5+0.1\sin(t)$ (middle panel) for which $\langle \beta_m \rangle=0.5$. We observe that survival occurs only in the \lq\lq middle'' case, meaning that $c_{l}^*=c_r^*<2<c_m^*$, or equivalently that $\lambda_m<\lambda_l=\lambda_r$, with obvious notations. Indeed, because of the Dirichlet boundary conditions, the optimum \lq\lq in the middle'' enhances the chances of survival. Note also that, because of the shift of the domain, the left and right cases are not symmetric any longer:   in the later case, the population is kept alive for a slightly longer period of time.
\begin{figure}[h!]
   \centering
\includegraphics[scale=0.25,height=4cm]{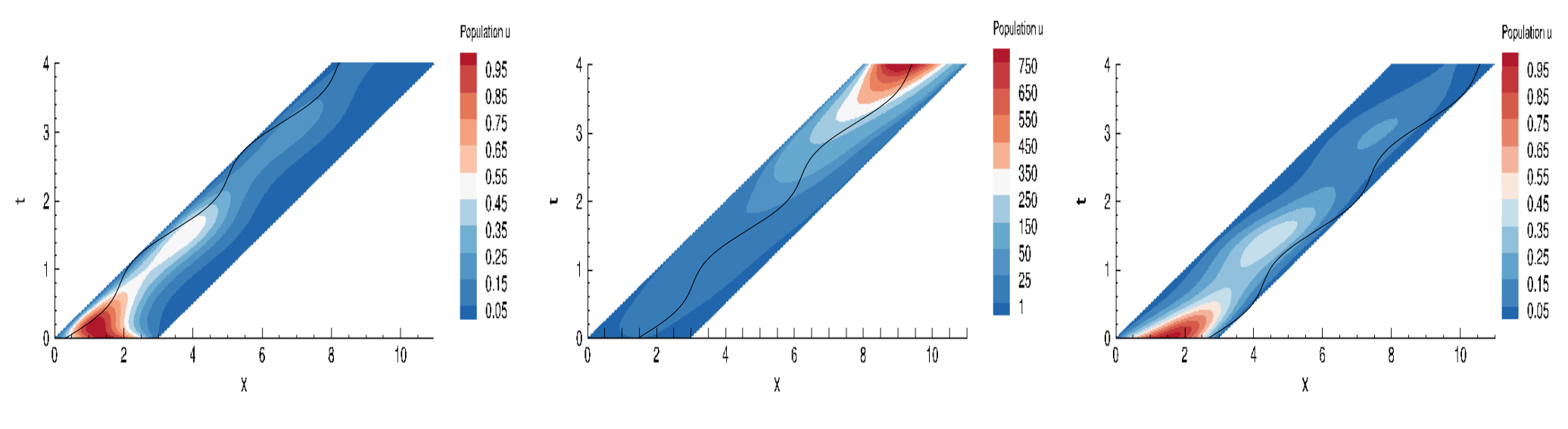}
\caption{\textbf{Effect of a linear domain shift and of the position of the optimum.} The parameters are as described in Example \ref{ex:linear}.  In solid black lines, the evolution of the position of the optimum $x_{opt}(t)=A(t)+\beta (t)L$.
}
\label{fig:lin}
\end{figure}
\end{example}
\begin{example}[Linear enlargement]\label{ex:enlarge}
 We  consider a domain that does not shift but gradually enlarges over time, namely  $A(t)=0$ and  either  $L(t)=1+0.3t$  (left panel of Figure \ref{fig:enlarge} with $T= 11$), or $L(t)=1+2t$ (right panel of Figure \ref{fig:enlarge} wit $T=15$). Other parameters are set as follows: $d=1$, $r=3.5$, $\alpha (t)=3+0.2\sin(2t)$, and $\beta(t)=0.8+0.1\sin(2t)$. The results are presented in Figure \ref{fig:enlarge}. We observe that the population survives if enlargement is slow (left panel) but goes to extinction if it is too fast (right panel). This may sound slightly counter-intuitive since one may expect  larger domains to be more fitted for survival. However, extinction is here likely caused by the optimal position shifting \lq\lq too fast'' to the right.
\begin{figure}[h!]
\centering
\includegraphics[scale=0.7,height=4.5cm]{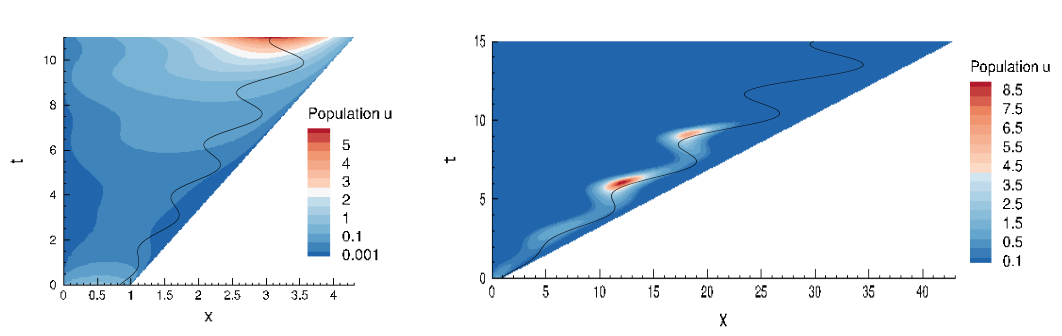}  
\caption{\textbf{Effect of the speed of enlargement.} 
The parameters are as described in Example \ref{ex:enlarge}.  In solid black lines, the evolution of the position of the optimum $x_{opt}(t)=\beta (t)L(t)$.
}
\label{fig:enlarge}
\end{figure}
\end{example}
\appendix
\section{ Appendix: Periodic parabolic eigenproblems}\label{s:appendix}
The following result on periodic parabolic eigenelements is borrowed from \cite[Theorem 1]{Cas-Laz-82}.
\begin{theorem}[Periodic parabolic eigenelements]\label{th:Cas-Laz-82} 
Let $\Omega \subset \R^n$ be an open bounded domain with boundary of class $\mathcal{C}^{2+\nu}$ for some $\nu \in (0,1) $. Let $L$ be the differential operator
\begin{equation*}
Lu:=u_t-\sum^{n}_{i,j=1}a_{ij}(t,x)u_{x_i,x_j}-\sum^{n}_{i=1}b_{i}(t,x)u_{x_i}-c(t,x)u
\end{equation*}
where all the coefficients $a_{ij}$, $b_{i}$, $c$ belong to $\mathcal{C}^{\frac \nu 2,\nu}(\R\times \overline{\Omega})$ and are $T$-periodic in time. Assume that $L$ is uniformly parabolic, namely there is $m>0$ such that
$$
\forall (t,x)\in \R \times \overline \Omega,\, \forall (\xi_1,\cdots,\xi_n)\in \R^n,\,  \sum_{i,j=1}^n a_{ij}(t,x)\xi_i\xi_j\geq m \sum_{i=1}^n \xi_i^2.
$$ 

Then there is a unique  $\lambda\in \R$ (the principal eigenvalue) such that there is a unique (up to multiplication by a positive constant) function $\varphi \in \mathcal{C}^{1+\frac \nu 2,2+\nu}(\R\times \overline \Omega)$ (the principal eigenfunction) solving
\begin{equation*}
\begin{cases}
L\varphi=\lambda \varphi   & \quad t\in \R, \, x\in \Omega,\vspace{5pt}\\
\varphi(t,x)=0, &\quad t\in \R,\, x\in \partial \Omega, \vspace{5pt}\\
\varphi>0 & \quad  t\in \R, \, x\in\Omega, \vspace{5pt} \\
\varphi(t,x)=\varphi(t+T,x) & \quad  t\in \R, \, x\in \Omega.
\end{cases}
\end{equation*}
Furthermore, if $n$ denotes the unit outer normal to $\partial \Omega$ at $x$ then 
\begin{equation}
\label{Hopf}
\frac{\partial \varphi}{\partial n}(t,x) <0, \quad \text{ for all } t\in \R,\, x\in \partial \Omega.
\end{equation}
\end{theorem}

\medskip

\noindent{\bf Acknowledgement.}  Matthieu Alfaro is grateful to the team RAPSODI (INRIA Lille) for its hospitality and nice atmosphere for work. Matthieu Alfaro is supported by  the  ANR project DEEV ANR-20-CE40-0011-01. The second author would like to thank Ch\'erif Amrouche for insightful and fruitful discussions on trace theorems.

\bibliographystyle{siam}  

\bibliography{biblio}

\begin{thebibliography}{10}

\bibitem{AinsworthOden2000}
{\sc J.~T. Ainsworth, M.and~Oden}, {\em A posteriori error estimation in finite
  element analysis}, Wiley-Interscience, 2000.

\bibitem{Alf-Ber-Rao-17}
{\sc M.~Alfaro, H.~Berestycki, and G.~Raoul}, {\em The effect of climate shift
  on a species submitted to dispersion, evolution, growth, and nonlocal
  competition}, SIAM J. Math. Anal., 49 (2017), pp.~562--596.

\bibitem{Alf-Car-17}
{\sc M.~Alfaro and R.~Carles}, {\em Replicator-mutator equations with quadratic
  fitness}, Proc. Amer. Math. Soc., 145 (2017), pp.~5315--5327.

\bibitem{Alf-Ham-Pat-Roq-23}
{\sc M.~Alfaro, F.~Hamel, F.~Patout, and L.~Roques}, {\em Adaptation in a
  heterogeneous environment {II}: to be three or not to be}, J. Math. Biol., 87
  (2023), pp.~Paper No. 68, 53.

\bibitem{Alf-Ver-18}
{\sc M.~Alfaro and M.~Veruete}, {\em Evolutionary branching via
  replicator--mutator equations}, Journal of Dynamics and Differential
  Equations, 31 (2019), pp.~2029--2052.

\bibitem{All-2}
{\sc J.~Allwright}, {\em Exact solutions and critical behaviour for a linear
  growth-diffusion equation on a time-dependent domain}, Proceedings of the
  Edinburgh Mathematical Society, 65 (2022), pp.~53--79.

\bibitem{All-3}
\leavevmode\vrule height 2pt depth -1.6pt width 23pt, {\em Reaction-diffusion
  problems on time-periodic domains}, arXiv preprint arXiv:2210.11516,  (2022).

\bibitem{All-1}
\leavevmode\vrule height 2pt depth -1.6pt width 23pt, {\em Reaction-diffusion
  on a time-dependent interval: refining the notion of `critical length'},
  Commun. Contemp. Math., 25 (2023), pp.~Paper No. 2250050, 11.

\bibitem{Ber-Die-Nag-Zeg-09}
{\sc H.~Berestycki, O.~Diekmann, C.~J. Nagelkerke, and P.~A. Zegeling}, {\em
  Can a species keep pace with a shifting climate?}, Bull. Math. Biol., 71
  (2009), pp.~399--429.

\bibitem{Bou-Gil-19}
{\sc J.~Bouhours and T.~Giletti}, {\em Spreading and vanishing for a monostable
  reaction-diffusion equation with forced speed}, J. Dynam. Differential
  Equations, 31 (2019), pp.~247--286.

\bibitem{Bou-Nad-15}
{\sc J.~Bouhours and G.~Nadin}, {\em A variational approach to
  reaction-diffusion equations with forced speed in dimension 1}, Discrete
  Contin. Dyn. Syst., 35 (2015), pp.~1843--1872.

\bibitem{Hughes1982}
{\sc A.~N. Brooks and T.~J.~R. Hughes}, {\em Streamline
  upwind/{P}etrov-{G}alerkin formulations for convection dominated flows with
  particular emphasis on the incompressible {N}avier-{S}tokes equations},
  Comput. Methods Appl. Mech. Engrg., 32 (1982), pp.~199--259.
\newblock FENOMECH ''81, Part I (Stuttgart, 1981).

\bibitem{Bur-00-book}
{\sc R.~B\"urger}, {\em The mathematical theory of selection, recombination,
  and mutation}, Wiley Series in Mathematical and Computational Biology, John
  Wiley \& Sons, Ltd., Chichester, 2000.

\bibitem{Can-Cos-96}
{\sc R.~S. Cantrell and C.~Cosner}, {\em Practical persistence in ecological
  models via comparison methods}, Proc. Roy. Soc. Edinburgh Sect. A, 126
  (1996), pp.~247--272.

\bibitem{Cas-Laz-82}
{\sc A.~Castro and A.~C. Lazer}, {\em Results on periodic solutions of
  parabolic equations suggested by elliptic theory}, Boll. Unione Mat. Ital.,
  VI. Ser., B, 1 (1982), pp.~1089--1104.

\bibitem{ErnGuermond2004}
{\sc A.~Ern and J.-L. Guermond}, {\em Theory and Practice of Finite Elements},
  Springer, 2004.

\bibitem{Evans2013}
{\sc J.~A. Evans and T.~J.~R. Hughes}, {\em Explicit trace inequalities for
  isogeometric analysis and parametric hexahedral finite elements}, Numerische
  Mathematik, 123 (2013), pp.~259--290.

\bibitem{Fig-Mir-18}
{\sc S.~Figueroa~Iglesias and S.~Mirrahimi}, {\em Long time evolutionary
  dynamics of phenotypically structured populations in time-periodic
  environments}, SIAM J. Math. Anal., 50 (2018), pp.~5537--5568.

\bibitem{Fig-Mir-21}
\leavevmode\vrule height 2pt depth -1.6pt width 23pt, {\em Selection and
  mutation in a shifting and fluctuating environment}, Commun. Math. Sci., 19
  (2021), pp.~1761--1798.

\bibitem{Geymonat2000}
{\sc G.~Geymonat and F.~c. Krasucki}, {\em On the existence of the {A}iry
  function in {L}ipschitz domains. {A}pplication to the traces of {$H^2$}}, C.
  R. Acad. Sci. Paris S\'er. I Math., 330 (2000), pp.~355--360.

\bibitem{Grisvard1985}
{\sc P.~Grisvard}, {\em Elliptic problems in nonsmooth domains}, Pitman
  advanced Publishing Program, 1985.

\bibitem{Ham-Lav-Mar-Roq-20}
{\sc F.~Hamel, F.~Lavigne, G.~Martin, and L.~Roques}, {\em Dynamics of
  adaptation in an anisotropic phenotype-fitness landscape}, Nonlinear Anal.
  Real World Appl., 54 (2020), pp.~103107, 33.

\bibitem{Ham-Lav-Roq-21}
{\sc F.~Hamel, F.~Lavigne, and L.~Roques}, {\em Adaptation in a heterogeneous
  environment {I}: persistence versus extinction}, J. Math. Biol., 83 (2021),
  pp.~Paper No. 14, 42.

\bibitem{book-Hes-91}
{\sc P.~Hess}, {\em Periodic-parabolic boundary value problems and positivity},
  vol.~247 of Pitman Research Notes in Mathematics Series, Longman Scientific
  \& Technical, Harlow; copublished in the United States with John Wiley \&
  Sons, Inc., New York, 1991.

\bibitem{Hughes1987}
{\sc T.~Hughes}, {\em The Finite Element Method: Linear Static and Dynamic
  Finite Element Analysis}, Dover Publications, 1987.

\bibitem{Hughes1986}
{\sc T.~J.~R. Hughes, L.~P. Franca, and M.~Balestra}, {\em A new finite element
  formulation for computational fluid dynamics: V. circumventing the "{K}elly's
  condition" for stable and accurate solutions}, Computer Methods in Applied
  Mechanics and Engineering, 58(1), 1-23.,  (1986).

\bibitem{Hut-95}
{\sc V.~Hutson, J.~L\'{o}pez-G\'{o}mez, K.~Mischaikow, and G.~Vickers}, {\em
  Limit behaviour for a competing species problem with diffusion}, in Dynamical
  systems and applications, vol.~4 of World Sci. Ser. Appl. Anal., World Sci.
  Publ., River Edge, NJ, 1995, pp.~343--358.

\bibitem{Hut-She-Vic-01}
{\sc V.~Hutson, W.~Shen, and G.~T. Vickers}, {\em Estimates for the principal
  spectrum point for certain time-dependent parabolic operators}, Proc. Amer.
  Math. Soc., 129 (2001), pp.~1669--1679.

\bibitem{Johnson1987}
{\sc C.~Johnson}, {\em Numerical Solution of Partial Differential Equations by
  the Finite Element Method}, Cambridge University Press, 1987.

\bibitem{Kim-65}
{\sc M.~Kimura}, {\em {A stochastic model concerning the maintenance of genetic
  variability in quantitative characters.}}, Proc. Natl. Acad. Sci. USA, 54
  (1965), pp.~731--736.

\bibitem{Lan-75}
{\sc R.~Lande}, {\em The maintenance of genetic variability by mutation in a
  polygenic character with linked loci}, Genetics Research, 26 (1975),
  pp.~221--235.

\bibitem{Lav-23}
{\sc F.~Lavigne}, {\em Adaptation of an asexual population with environmental
  changes}, Math. Model. Nat. Phenom., 18 (2023), pp.~Paper No. 20, 21.

\bibitem{Liu-Lou-Pen-Zho-19}
{\sc S.~Liu, Y.~Lou, R.~Peng, and M.~Zhou}, {\em Monotonicity of the principal
  eigenvalue for a linear time-periodic parabolic operator}, Proc. Amer. Math.
  Soc., 147 (2019), pp.~5291--5302.

\bibitem{Liu-Lou-Pen-Zho-21}
\leavevmode\vrule height 2pt depth -1.6pt width 23pt, {\em Asymptotics of the
  principal eigenvalue for a linear time-periodic parabolic operator {II}:
  {S}mall diffusion}, Trans. Amer. Math. Soc., 374 (2021), pp.~4895--4930.

\bibitem{lohwater2013boundary}
{\sc J.~Lohwater and O.~Ladyzhenskaya}, {\em The Boundary Value Problems of
  Mathematical Physics}, Applied Mathematical Sciences, Springer New York,
  2013.

\bibitem{Mar-Len-15}
{\sc G.~Martin and T.~Lenormand}, {\em The fitness effect of mutations across
  environments: Fisher's geometrical model with multiple optima}, Evolution, 69
  (2015), pp.~1433--1447.

\bibitem{Moore_2018}
{\sc S.~E. Moore}, {\em A stable space--time finite element method for
  parabolic evolution problems}, Calcolo, 55 (2018), p.~18.

\bibitem{Nad-09}
{\sc G.~Nadin}, {\em The principal eigenvalue of a space-time periodic
  parabolic operator}, Ann. Mat. Pura Appl. (4), 188 (2009), pp.~269--295.

\bibitem{Pot-Lew-04}
{\sc A.~B. Potapov and M.~A. Lewis}, {\em Climate and competition: the effect
  of moving range boundaries on habitat invasibility}, Bull. Math. Biol., 66
  (2004), pp.~975--1008.

\bibitem{Roq-Pat-Bon-Mar-20}
{\sc L.~Roques, F.~Patout, O.~Bonnefon, and G.~Martin}, {\em Adaptation in
  general temporally changing environments}, SIAM Journal on Applied
  Mathematics, 80 (2020), pp.~2420--2447.

\bibitem{Roq-Roq-Ber-Kre-08}
{\sc L.~Roques, A.~Roques, H.~Berestycki, and A.~Kretzschmar}, {\em A
  population facing climate change: joint influences of allee effects and
  environmental boundary geometry}, Population Ecology, 50 (2008),
  pp.~215--225.

\bibitem{Ten-14}
{\sc O.~Tenaillon}, {\em {The utility of Fisher's geometric model in
  evolutionary genetics}}, Annual Review of Ecology, Evolution, and
  Systematics, 45 (2014), pp.~179--201.

\end{thebibliography}

\end{document}